\documentclass[a4paper]{amsart} 
\usepackage[utf8]{inputenc}
\usepackage[ngerman, english]{babel}
\usepackage{amsmath,amsthm,amssymb,amsfonts}
\usepackage{mathrsfs}
\usepackage{ifthen}
\usepackage{enumerate}
\usepackage{hyphsubst}
\usepackage{xcolor}
\usepackage{graphicx}
\usepackage{psfrag}
\usepackage{pst-all}
\usepackage{xcolor}

\newtheorem{theorem}{Theorem}[section]
\newtheorem{lemma}[theorem]{Lemma}

\newtheorem{definition}[theorem]{Definition}
\newtheorem{assumption}[theorem]{Assumption}

\newcommand{\Qpol}{Q}
\newcommand{\sph}{S^2}
\DeclareMathOperator{\Px}{P_x}
\DeclareMathOperator{\I}{I}
\DeclareMathOperator{\D}{D}
\newcommand{\spm}{\omega}
\newcommand{\spml}{\lambda}
\newcommand{\cofi}{\chi_1}
\newcommand{\cofii}{\chi_2}

\newcommand{\inde}{e}
\newcommand{\exr}{r_\inde}
\newcommand{\exx}{x_\inde}
\newcommand{\exF}{F_\inde}
\newcommand{\exgamma}{\gamma_\inde}
\newcommand{\extgamma}{\tilde\gamma_\inde}
\newcommand{\extr}{\tilde r_\inde}
\newcommand{\extalpha}{\tilde\alpha_\inde}
\newcommand{\exalpha}{\alpha_\inde}
\newcommand{\extd}{\tilde d_\inde}
\newcommand{\exd}{d_\inde}
\newcommand{\exhatd}{\hat d_\inde}
\newcommand{\exOmega}{\Omega_\inde}
\newcommand{\exX}{X_\inde}
\newcommand{\exXh}{X_\inde^{h(n)}}
\newcommand{\exa}{a_\inde}
\newcommand{\exA}{A_\inde}
\newcommand{\exAh}{A_{\inde,h(n)}}
\newcommand{\exeta}{\eta_{\inde}}
\newcommand{\exetaeps}{\eta_{\inde,\epsilon}}
\newcommand{\exT}{T_\inde}
\newcommand{\exTeps}{T_{\inde,\epsilon}}
\newcommand{\exTepsz}{T_{\inde,{\epsilon_0}}}
\newcommand{\exTepszn}{T_{\inde,{\epsilon_0}}^{h(n)}}
\newcommand{\exPin}{\Pi^\inde_{h(n)}}

\newcommand{\dd}{\mathrm{d}}
\newcommand{\ol}[1]{\overline{#1}}

\newcommand{\spl}{\langle}
\newcommand{\spr}{\rangle}
\newcommand{\bpm}{\begin{pmatrix}}
\newcommand{\epm}{\end{pmatrix}}

\renewcommand{\div}{\operatorname{div}}

\renewcommand{\dim}{\operatorname{dim}}

\newcommand{\setC}{\mathbb{C}}

\newcommand{\setN}{\mathbb{N}}

\newcommand{\setR}{\mathbb{R}}

\usepackage[pdfauthor={Martin Halla},pdftitle={Analysis of radial complex scaling methods: scalar resonance problems},
unicode,colorlinks=true,linkcolor=black,citecolor=black,urlcolor=black,pagebackref=false,breaklinks]{hyperref}

\title[Analysis of radial complex scaling methods]{Analysis of radial complex scaling methods: scalar resonance problems}
\author{Martin Halla}
\email{halla@mps.mpg.de}
\date{\today}

\begin{document}
\begin{abstract}
We consider radial complex scaling/perfectly matched layer methods for scalar resonance problems in homogeneous exterior domains.
We introduce a new abstract framework to analyze the convergence of domain truncations and discretizations.
Our theory requires rather minimal assumptions on the scaling profile and includes affin, smooth and also unbounded profiles.
We report a swift technique to analyze the convergence of domain truncations and a more technical one for approximations through simultaneaous truncation and discretization.
We adapt the latter technique to cover also so-called exact methods which do not require a domain truncation.
Our established results include convergence rates of eigenvalues and eigenfunctions.

The introduced framework is based on the ideas to interpret the domain truncation as Galerkin approximation, to apply theory on holomorphic Fredholm operator eigenvalue approximation theory to a linear eigenvalue problem, to employ the notion of weak T-coercivity and T-compatible approximations, to construct a suitable T-operator as multiplicatin operator, to smooth its symbol and to apply the discrete commutator technique.
\end{abstract}
\maketitle

\section{Introduction}\label{sec:introduction}

Since the 1970s a popular method has been used in molecular physics to study resonances \cite{Simon:78,Moiseyev:98}. This method is referred to by various names: complex scaling (CS), analytic dilation (AD) and spectral deformation (SD). It admits a profound mathematical framework with the Aguilar-Balslev-Combes-Simon Theorem at its core~\cite{HislopSigal:96}.
A main advantage of the method is that it preserves the \emph{linear} eigenvalue problem structure. In contrast other types of transparent boundary conditions such as absorbing boundary conditions \cite{Givoli:92,Hagstrom:99} do not work for resonance problems or destroy the linear nature like boundary element methods \cite{Unger:09,SteinbachUnger:09,Unger:17}.
In the 1990s B\'erenger \cite{Berenger:94} introduced his perfectly matched layer (PML) method as reflectionless sponge layer for electromagnetic scattering problems which became very popular for all kinds of wave propagation problems. In \cite{ChewWeedon:94,TeixeiraChew:97,CollinoMonk:98a} the PML method was recognized to be a complex scaling technique.
Also the original variant of the Hardy space infinite element method \cite{Nannen:08,HohageNannen:09,Halla:16} was recognized in \cite{NannenWess:19} to be a conjunction of a complex scaling and an infinite element method.
We refer to \cite{HeinHohageKoch:04,HeinHohageKochSchoeberl:07} for computational studies
of CS/PML methods for resonance problems.
Recently CS/PML methods have been applied to other kinds of problems as well. E.g.\ problems which are posed in bounded domains, but admit black hole phenomena \cite{BonnetBDCarvalhoChesnelCiarlet:16}. A further application of
CS/PML can be found in domain decomposition methods \cite{EngquistYing:11,LiuYing:16}.
We refer to the introduction of \cite{Halla:19Diss} for a rigorous overview on the existing literature on CS/PML methods.

The idea of CS/PML methods is to apply a continuous complex coordinate transformation (the complex scaling) to the resonance functions. For resonances in a suitable region of the complex plane the corresponding resonance functions become exponentially damped by
this transformation. Consequently a new set of partial differential equations is derived for the transformed resonance functions and due to their exponential decay the resonance problem transforms to an eigenvalue problem in a suitable standard Sobolev space.
Furthermore the domain is truncated to a bounded one and a homogeneous boundary condition imposed at the artificial boundary. Due to the rapid decay of the eigenfunctions the committed error is expected to be small. The derived problem can consequently be discretized with standard numerical schemes such as finite element methods.

We note that despite their popularity the construction of physically correct and stable CS/PML methods has to be executed with care.
In general the complex scaling has to be designed so that evanescent waves stay evanescent and propagative waves with positive group velocity become evanescent.
While for some equations this poses no problem at all it can lead to serious difficulties if the equation is anisotropic \cite{BecacheFauqueuxJoly:03}, advective \cite{BecacheBonnetBDLegendre:04,BecacheBonnetBDLegendre:06} or dispersiv \cite{CassierJolyKachanovska:17,BecacheKachanovska:17}. Also waveguide geometries \cite{SkeltonAdamsCraster:07,BonnetBDChambeyronLegendre:14,HallaNannen:15,HallaHohageNannenSchoeberl:16} can generate such difficulties since they cause dispersive effects (although the equation may be dispersionless itself).

While the application of the CS/PML method is widely spread, the results on proofs of convergence have been limited so far. For time-dependent equations the only actual convergence result known to us is \cite{Diaz:05,DiazJoly:06}. Time-harmonic scattering problems have been dealt with by a number of authors (e.g.\ \cite{LassasSomersalo:98,HohageSchmidtZschiedrich:03b,BaoWu:05,BramblePasciak:07,BramblePasciak:08a,BramblePasciakTrenev:10,ChenXiangZhang:16}). However, the results are usually formulated for special scaling profiles and a unified framework was missing. For resonance problems the only known convergence results known to use are \cite{Kim:09,KimPasciak:09,Kim:14,HohageNannen:15}. Moreover with the exception of \cite{HohageNannen:15} all works analyze the domain truncation and the subsequent (finite element) discretization seperately. Consequently the important question if arbitrary combinations of domain truncations and discretizations can lead to erroneous results is left open.

This article considers radial CS/PML approximations for scalar resonance problems in homogeneous exterior domains.
We introduce a new abstract framework for the convergence analysis which can also serve for the analysis of other equations and geometries.
On the one hand we present a swift technique to analyze the domain truncation and the subsequent discretization seperately. On the other hand we also cover simultaneaous approximations (with more technical effort). Further, we adapt our analysis to treat so-called exact methods \cite{HugoninLalanne:05,BermudezHervellaNPrietoRodriguez:08} which do not require a domain truncation. Different to existing works \cite{Kim:09,KimPasciak:09,Kim:14,HohageNannen:15} we also report convergence rates of eigenvalues and eigenfunctions. Our results are formulated under rather minimal assumptions on the profile function and cover affine, smooth as well as unbounded profiles.

Our framework is build on the combination of several indepent ideas.
Following \cite{HohageNannen:15} we interpret the domain truncation as Galerkin approximation.
This restores a most convenient setup to perform the approximation analysis.
In addition to \cite{HohageNannen:15} we propose to use this idea also for the analysis of the sole domain truncation (without discretization).
Due to their large essential spectrum resonance problems can't be reformulated as standard linear eigenvalue problems for a compact operator and standard eigenvalue approximation theory \cite{BabuskaOsborn:91,Boffi:10} cannot be applied. Following \cite{Halla:16} we apply literature on holomorphic Fredholm operator eigenvalue approximation theory \cite{Karma:96a,Karma:96b} to a linear eigenvalue problem. This allows us to employ readily available concepts and relieves us from conducting huge parts of the analysis manually (as done in \cite{Kim:09,KimPasciak:09,Kim:14,HohageNannen:15}). This way we also easily obtain  convergence rates.
We ensure the regularity/stability of Galerkin approximations through the notion of weak T-coercivity and T-compatible approximations \cite{Halla:19Tcomp}. 
For our eigenvalue problem at hand the construction of a suitable T-operator can be realized with a simple multiplication operator. For restricted kinds of scaling profiles this observation goes back to \cite{BramblePasciak:07}.
To treat simultaneaous approximations and exact methods we smooth the symbol of the multiplication operator and employ the discrete commutator technique of \cite{Bertoluzza:99}.

The remainder of this article is structured as follows.
In Section~\ref{sec:theproblem} we introduce the original resonance problem, the CS/PML eigenvalue problem and discuss their relation.
We further introduce the domain truncation as Galerkin approximation and discuss why convenient theory on linear eigenvalue problem approximations cannot be applied.
In Section~\ref{sec:framework} we recall the T-analysis framework \cite{Halla:19Tcomp}, explicitly construct a T-operator and establish a weak T-coercivity result in Theorem~\ref{thm:AwTc}.
In Section~\ref{sec:appr_subseq} we discuss approximation through domain truncations.
We establish in Theorem~\ref{thm:SpectralConvergenceHelmholtz} convergence and explain how to obtain convenient exponential error estimates by means of Lemma~\ref{lem:BestApproximation}.
In Section~\ref{sec:appr_simul} we report similar results for the more subtle case of simultaneaous domain truncation and discretization.
In Section~\ref{sec:appr_truncationless} we introduce a reformulation of the CS/PML eigenvalue problem on a bounded domain and establish convergence of approximations in Theorem~\ref{thm:SpectralConvergenceHelmholtzExact}.
We discuss how to choose the CS/PML parameters such that common finite element methods fit into the former theory.
In Section~\ref{sec:conclusion} we conclude and discuss the perspective to generalize the presented results to other equations and geometric configurations.

\section{The resonance problem and its approximation}\label{sec:theproblem}

\subsection{The resonance problem}\label{subsec:resonanceproblem}

Let $B_r\subset\setR^3$ be the open ball with radius $r>0$ centered at the origin,
$B_r(x_0)\subset\setR^3$ be the open ball with radius $r>0$ centered at $x_0$
and $A_{r_1,r_2}\subset\setR^3$ be the open annulus $B_{r_2}\setminus \ol{B_{r_1}}$ with radii $r_2>r_1>0$.
For a Lipschitz domain $D\subset\setR^3$ let
\begin{align}
\tilde H^1_\mathrm{loc}(D)&:=\{u\in H^1_\mathrm{loc}(D)\colon u|_{D\cap B_r}\in H^1(D\cap B_r)
\text{ for all }r>0\text{ with }D\cap B_r\neq\emptyset\}.
\end{align}
For a Lipschitz domain $D\subset\setR^3$ with finite boundary $\partial D$ and $u\in H^1_\mathrm{loc}(D)$ the trace
$u|_{\partial D}\in H^{1/2}(\partial D)$ is well defined. Hence let
\begin{align}
H^1_\mathrm{0,loc}(D)&:=\{u\in \tilde H^1_\mathrm{loc}(D)\colon u|_{\partial D}=0\}.
\end{align}
Let $\Omega\subset\setR^3$ be a Lipschitz domain so that the complement $\Omega^c$ is compact and non-empty.
We seek non-trivial solutions $(\spm,u)$ to
\begin{subequations}\label{eq:ResonanceProblemFormal}
\begin{align}
\label{eq:ResonanceProblemFormalA} -\Delta u-\spm^2 u&=0\quad\text{in }\Omega,\\
\label{eq:ResonanceProblemFormalB} u&=0\quad\text{at }\partial\Omega,
\end{align}
together with the abstract radiation condition (which will be specified in
Definition~\ref{def:RadiationCondition})
\begin{align}
u&\text{ is outgoing}
\end{align}
\end{subequations}
in the distributional sense. That is $(\spm,u)$ solves
\begin{subequations}\label{eq:ResonanceProblemDistributional}
\begin{align}
\label{eq:ResonanceProblemDistributionalA}
\text{find }(\spm,u)\in\setC\setminus\{0\}\times \tilde H^1_\mathrm{0,loc}(\Omega)\setminus\{0\}&\text{ such that}\nonumber\\
\langle \nabla u,\nabla u'\rangle_{L^2(\Omega)}-\spm^2\langle u,u'\rangle_{L^2(\Omega)}&=0
\quad\text{for all }u'\in C^\infty_{0}(\Omega),\\
u&\text{ is outgoing}.
\end{align}
\end{subequations}

Let $S^2_r:=\{x\in\setR^3\colon|x|=r\}$ be the sphere with radius $r>0$ and $S^2:=S^2_1$ be the unit sphere.
Consider the standard parametrization
\begin{align}
\Qpol(r,\hat x):=r\hat x, \quad r>0,\hat x\in S^2.
\end{align}
It is well known (see e.g. \cite[Lemmata~2.3 and~2.4]{Halla:19Diss}) that every solution to the Helmholtz equation $-\Delta u-\spm^2 u=0$
in an annulus $A_{r_1,r_2}$ can be expanded in a series of tensor product functions of spherical Hankel functions of the
first $h^1_n(r)$ and second $h^2_n(r)$ kind and spherical harmonics $Y_n^m(\hat x)$. The meaningful physical radiation condition demands that
no spherical Hankel functions of the second kind occur.

\begin{definition}[Radiation condition]\label{def:RadiationCondition}
Let $(\spm,u)\in\setC\setminus\{0\}\times \tilde H^1_\mathrm{0,loc}(\Omega)\setminus\{0\}$ be a solution
to~\eqref{eq:ResonanceProblemDistributionalA}. We call $u$ to be outgoing if it admits a representation
\begin{align}\label{eq:RadiationCondition}
u\circ\Qpol(r,\hat x) = \sum_{n=0}^\infty\sum_{m=-n}^n a_n^mh^1_n(\spm r)Y_n^m(\hat x)
\end{align}
in $L^2(A_{r_1,r_2})$ for all $0<r_1<r_2$ with $\Omega^c\subset B_{r_1}$.
\end{definition}

\subsection{The complex scaled eigenvalue problem}\label{subsec:csevp}

We will define a complex change of the radial coordinate
$\tilde r(r)=(1+i\tilde\alpha(r))r$ in terms of a profile function
$\tilde\alpha$. We make assumptions on this profile function as follows.
\begin{assumption}\label{ass:TildeAlpha}
Let $r^*_1>0$ be such that $\Omega^c$ is contained in the ball $B_{r^*_1}$ and
$\tilde\alpha\colon\setR^+_0\to\setR^+_0$ be such that
\begin{enumerate}
 \item\label{it:talpha0} $\tilde\alpha(r)=0$ for $r\leq r^*_1$,
 \item\label{it:talphacont} $\tilde\alpha$ is continuous,
 \item\label{it:nontrivial} $\tilde\alpha(r)>0$ for $r>r^*_1$,
 \item\label{it:talphamonoton} $\tilde\alpha$ is non-decreasing,
 \item\label{it:talphaC2} $\tilde\alpha$ is twice continuously differentiable in $(r^*_1,+\infty)$ with continuous
 extensions of $\tilde\alpha$, $\partial_r\tilde\alpha$, $\partial_r\partial_r\tilde\alpha$ to $[r^*_1,+\infty)$.
\end{enumerate}
\end{assumption}

Assumption~\ref{ass:TildeAlpha} is very general.
Later on we will require an additional Assumption~\ref{ass:TildeAlpha2} for our analysis.
In particular Assumptions~\ref{ass:TildeAlpha} and~\ref{ass:TildeAlpha2} are satisfied by profiles of the following kinds. The probably simplest complex scaling is
\begin{subequations}
\begin{align}
\tilde r(r)=r+i\alpha_0(r-r_1^*), \quad r\geq r_1^*
\end{align}
with a constant $\alpha_0>0$. It corresponds to
\begin{align}\label{eq:AlphaAffin}
\tilde\alpha_\mathrm{affin}(r):=\alpha_0(1-r_1^*/r), \quad r\geq r_1^*.
\end{align}
\end{subequations}
A popular choice of complex scalings are power functions
\begin{subequations}
\begin{align}
\tilde r(r)=r+i\alpha_0(r-r_1^*)^m, \quad r\geq r_1^*
\end{align}
with a constant $\alpha_0>0$ and $m\in\setN$. They correspond to
\begin{align}\label{eq:AlphaPower}
\tilde\alpha_\mathrm{power}(r)&:=\alpha_0(r-r_1^*)^m/r, \quad r\geq r_1^*
\end{align}
\end{subequations}
with a constant $\alpha_0>0$. A profile which is more or less motivated by the
aim to simplify analysis is
\begin{align}\label{eq:AlphaSmooth}
\begin{split}
&\tilde\alpha_\mathrm{smooth}\text{ non-decreasing and twice continuous differentiable in }\setR^+,\\
&\tilde\alpha_\mathrm{smooth}(r):=0 \quad\text{for}\quad r\leq r_1^* \quad\text{and}\quad
\tilde\alpha_\mathrm{smooth}(r):=\alpha_0 \quad\text{for}\quad r\geq r_2,
\end{split}
\end{align}
with constants $\alpha_0>0$, $r_2>r_1^*$. In particular, many authors (e.g.\
\cite{LassasSomersalo:98}, \cite{HohageSchmidtZschiedrich:03b},
\cite{BramblePasciak:07}, \cite{KimPasciak:09}) only consider profiles of the last kind
for their analysis. An infinitely many times differentiable example of Kind~\eqref{eq:AlphaSmooth} is
\begin{align}\label{eq:AlphaInfty}
\tilde\alpha_\infty(r):=\alpha_0 \cofii(r-r_1^*)
\end{align}
with constant $\alpha_0>0$, $r_2=r_1^*+1$ and
\begin{subequations}
\begin{align}
\label{eq:Smooth0toInfty}
\cofi(r)&:=\left\{\begin{array}{rl} 
0, &\text{for }r\leq 0,\\
\exp(-1/r), &\text{for }r>0,
\end{array}\right.\\
\label{eq:Smooth0to1}
\cofii(r)&:=\left\{\begin{array}{rl}
0, &\text{for }r\leq 0,\\
\frac{\cofi(r)}{\cofi(r)+\cofi(1-r)}, &\text{for }0<r<1,\\
1, &\text{for }r\geq 1,
\end{array}\right.
.\end{align}
\end{subequations}

In the following we introduce additional functions which will all depend on $\tilde\alpha$.
These auxiliary functions will be necessary to formulate the forthcoming theory.
We adopt the notation of Bramble and Pasciak~\cite{BramblePasciak:07}. Hence let
\begin{subequations}\label{eq:ComplexScalingQuantities}
\begin{align}
\label{eq:tdDef} \tilde d(r)&:=1+i\tilde\alpha(r),\\
\label{eq:trDef} \tilde r(r)&:=\tilde d(r)r,\\
\label{eq:alphaDef} \alpha(r)&:=r\partial_r\tilde\alpha(r)+\tilde\alpha(r),\\
\label{eq:dDef} d(r)&:=1+i\alpha(r),\\
\label{eq:dzDef} d_0&:=\displaystyle\lim_{r\to+\infty}(\tilde d(r)/|\tilde d(r)|),\\
\label{eq:PxDef} \Px(x)&:=|x|^{-2}xx^\top, \quad x\in\setR^3,
\end{align}
whereby $xx^\top$ denotes the dyadic product.
The definitions of $\alpha$ and $d$ have to be understood piece-wise. We note that the limes in~\eqref{eq:dzDef}
exists in $\setC$ due to Assumption~\ref{ass:TildeAlpha}.
The function $d$ is chosen such that $\partial_r\tilde r(r)=d(r)$. For $f=\tilde\alpha,\alpha,\tilde d,d,\tilde r$
we adopt the overloaded notation
\begin{align}\label{eq:OverloadedNotation}
f(x):=f(|x|), \quad x\in\setR^3.
\end{align}
\end{subequations}
Hence we write e.g.\ $f\circ\Qpol(r,\hat x)=f(r)$.\\

Consider a solution $(\spm,u)$ to~\eqref{eq:ResonanceProblemDistributional}. Formally we can define
$\tilde u\circ\Qpol(r,\hat x):=u\circ\Qpol(\tilde r(r),\hat x)$. Due to Assumption~\ref{ass:TildeAlpha},
Expansion~\eqref{eq:RadiationCondition} and the asymptotic behaviour of spherical Hankel functions
we expect that $\tilde u$ is exponentially decreasing with respect to $|x|$.
By means of the chain rule we can formally deduce that $(\spm,\tilde u)$ solves
$-\tilde\Delta \tilde u-\spm^2 \tilde u=0$ whereby
\begin{align}\label{eq:tlaplace-rhatx}
\tilde\Delta u\circ\Qpol := (\tilde dr)^{-2}d^{-1}\partial_r(\tilde d^2r^2d^{-1}\partial_r u\circ\Qpol)
+(\tilde dr)^{-2}\Delta_{\sph} u\circ\Qpol,
\end{align}
i.e.\
\begin{align}\label{eq:tlaplace-x}
\tilde\Delta u = (\tilde d^2d)^{-1}\div\big( (\tilde d^2d^{-1}\Px + d(\I-\Px))\nabla u\big),
\end{align}
whereby $\I$ denotes the three by three identity matrix.
Vice-versa we expect that for a solution $(\spm,\tilde u)$ to $-\tilde\Delta\tilde u-\spm^2\tilde u=0$ we can define
$u$ in reversal of $\tilde u$ and expect that $(\spm,u)$ solves $-\Delta u-\spm^2 u=0$.
However, since our coordinate transformation is complex valued this result is non-trivial.\\

We continue to formulate the respective eigenvalue problem for $\tilde u$.
For a Lipschitz domain $D\subset\Omega$ let
\begin{subequations}\label{eq:XD}
\begin{align}
X(D)&:=\{u\in \tilde H^1_\mathrm{0,loc}(D)\colon\langle u,u\rangle_{X(D)}<\infty\},\\
\langle u,u' \rangle_{X(D)}&:=\langle (|\tilde d^2d^{-1}|\Px+|d|(\I-\Px))\nabla u,\nabla u'\rangle_{L^2(D)}
+\langle |\tilde d^2d|u,u' \rangle_{L^2(D)}.
\end{align}
\end{subequations}
and
\begin{align}\label{eq:aD}
a_D(\spm; u,u'):=\langle (\tilde d^2d^{-1}\Px + d(\I-\Px))\nabla u,\nabla u'\rangle_{L^2(D)}
-\spm^2\langle \tilde d^2d u,u' \rangle_{L^2(D)}
\end{align}
for $\spm\in\setC$ and $u,u'\in X(D)$. By definition of $X(D)$ the sesquilinearform
$a_D(\spm;\cdot,\cdot)$ is bounded on $X(D)\times X(D)$. For $D=\Omega$ we set
\begin{align}\label{eq:X-sprX-a}
X:=X(\Omega), \qquad \langle \cdot,\cdot\rangle_X:=\langle \cdot,\cdot\rangle_{X(\Omega)},
\qquad a(\cdot;\cdot,\cdot):=a_\Omega(\cdot;\cdot,\cdot).
\end{align}
Consider the eigenvalue problem to
\begin{align}\label{eq:ResonanceProblemVariational}
\text{find }(\spm,\tilde u)\in\setC\times X\setminus\{0\}\text{ such that}\quad a(\spm;\tilde u,u')=0
\quad\text{for all }u'\in X.
\end{align}
Note that the introduced space $X$ is of importance only for profile functions with
$\tilde\alpha, \alpha\not\in L^\infty(\setR^+)$. Else wise $X$ is reduced to the standard
Sobolev space $H^1_0(\Omega)$ (equipped with an equivalent inner product).\\

We know from Theorem~2.16 of \cite{Halla:19Diss} that for each solution $(\spm,u)$ to~\eqref{eq:ResonanceProblemDistributional}
with $\Re(i\spm d_0)<0$ the function $\tilde u\circ\Qpol(r,\hat x):=u\circ\Qpol(\tilde r(r),\hat x)$ is well defined,
contained in $X$ and $(\spm,\tilde u)$ solves~\eqref{eq:ResonanceProblemVariational}.
Moreover, for any $\epsilon>0$ there exists a constant $C(\tilde u)>0$ so that
\begin{align}\label{eq:EstXNormU}
\|\tilde u\|_{X(B_r^c)}^2 \leq C(\tilde u) \int_r^{+\infty} e^{2(\Re(i\spm d_0)+\epsilon)\sqrt{1+\tilde\alpha(t)^2}t} \dd t
\quad\text{for all}\quad r\geq r_1^*.
\end{align}
Vice-versa if $(\spm,\tilde u)$ is a solution to~\eqref{eq:ResonanceProblemVariational} with
$\Re(i\spm d_0)<0$, then $u$ can be well defined in a reverse manner and $(\spm,u)$
solves~\eqref{eq:ResonanceProblemDistributional}.

\subsection{Domain truncation and discretization}\label{subsec:appr}

For a solution $(\spm,\tilde u)$ to~\eqref{eq:ResonanceProblemVariational} with
$\Re(i\spm d_0)<0$ it follows from \eqref{eq:EstXNormU} that $\tilde u$
decays exponentially to zero as $x\to\infty$. Thus it seems natural to
approximate~\eqref{eq:ResonanceProblemVariational} by replacing the domain
$\Omega$ with a bounded subdomain $\Omega_n$ and pose a homogeneous Dirichlet or
Neumann boundary condition at the artificial boundary
$\partial\Omega_n\setminus\partial\Omega$. As most authors we stick to Dirichlet
boundary conditions. The resulting equation can then be discretized with a
standard numerical scheme such as finite element methods. The question arises if
and also how fast the solutions to this approximation converge to the solutions
of the original Equation~\eqref{eq:ResonanceProblemVariational}.

It is a classical approach to separate the analysis into a truncation analysis and
a discretization analysis. This seperation allows to simplfy the analysis considerably.
In particular the analysis for the discretization step can be performed in the very same
manner as for classical problems posed on bounded domains.
However, such a seperated analysis does not ensure that every arbitrary (resonable)
sequence of combinations of domain truncations and finite element spaces yields a converging
approximation.

In the following we will explain why convenient techniques cannot be applied and therefore
introduce a new notion to perform the truncation analysis.
(Lateron, in Section~\ref{sec:appr_simul} we will also present a way to perform a
simultaneaous approximation analysis.)
To this end, we make our Assumptions on $\Omega_n$ more precise.
\begin{assumption}\label{ass:OmegaN}
The sequence of subdomains $(\Omega_n)_{n\in\setN}$ is such that for each $n\in\setN$
\begin{enumerate}
 \item $\Omega_n$ is a bounded Lipschitz domain,
 \item $\Omega_n\subset\Omega$,
 \item $\partial\Omega\subset\partial\Omega_n$,
 \item $\partial\Omega_n\setminus\partial\Omega$ splits $\Omega$ into two connected parts,
\end{enumerate}
and for any $R>0$ exists an index $n_0\in\setN$ such that $(\Omega\cap B_R)\subset \Omega_n$
for all $n>n_0$.
\end{assumption}
The PML approximation to~\eqref{eq:ResonanceProblemVariational} reads
\begin{align}\label{eq:ResonanceProblemPML}
\text{find }(\spm,u_n)\in\setC\times X(\Omega_n)\setminus\{0\}\text{ so that}
\quad a_{\Omega_n}(\spm;u_n,u_n')=0\quad\text{for all }u_n'\in X(\Omega_n).
\end{align}
We note that $\|\cdot\|_{X(\Omega_n)}$ is an equivalent norm to
$\|\cdot\|_{H^1(\Omega_n)}$ and hence $X(\Omega_n)=H^1_0(\Omega_n)$. Let
\begin{align}\label{eq:Xn}
X_n:=\{u\in X\colon u=0\text{ in }\Omega\setminus\Omega_n\}
\end{align}
and consider the problem to
\begin{align}\label{eq:ResonanceProblemPMLGalerkin}
\text{find }(\spm,u_n)\in\setC\times X_n\setminus\{0\}\text{ so that}
\quad a(\spm;u_n,u_n')=0\quad\text{for all }u_n'\in X_n.
\end{align}
It is obvious that for every solution $(\spm,u_n)$
to~\eqref{eq:ResonanceProblemPML} the extension $\hat u_n$ of $u_n$ to
$\Omega\setminus\Omega_n$ by zero is in $X_n$ and $(\spm,\hat u_n)$
solves~\eqref{eq:ResonanceProblemPMLGalerkin}. Vice-versa for every solution
$(\spm,u)$ to~\eqref{eq:ResonanceProblemPMLGalerkin} $(\spm,u|_{\Omega_n})$
solves~\eqref{eq:ResonanceProblemPML}. However as $X_n$ is a subspace of $X$ we
recognize~\eqref{eq:ResonanceProblemPMLGalerkin} as conform Galerkin
approximation to~\eqref{eq:ResonanceProblemVariational}, which restores a common
setup for numerical analysts. The former notion is motivated
by~\cite{HohageNannen:15} wherein certain finite element spaces are
considered directly as subspaces of $X$.
We note that the choice of Dirichlet boundary condition at the artificial
boundary is essential to ensure a conform approximation. Indeed an approximation
with Neumann boundary condition could be analyzed as non-conform approximation
to~\eqref{eq:ResonanceProblemVariational}. We will not continue further in this
direction as the analysis would be more intricate with barely additional gain.

Through our notion we can investigate the truncation error as a Galerkin error. A
classical way~\cite{BabuskaOsborn:91,Boffi:10} to analyze Galerkin
approximations to linear eigenvalue problems (such as ours) is to introduce
solution operators
\begin{align}\label{eq:SolutionOperatorA}
S\colon X\to X, \qquad S_n\colon X_n\to X_n
\end{align}
defined by
\begin{subequations}\label{eq:SolutionOperatorB}
\begin{align}
\label{eq:SolutionOperatorBi} a(1;Su,u')&=\langle \tilde d^2du,u' \rangle_{L^2(\Omega)}
&\quad\text{for all }u'\in X,\\
\label{eq:SolutionOperatorBii} a(1;S_nu_n,u_n')&=\langle \tilde d^2du_n,u_n' \rangle_{L^2(\Omega_n)}
&\quad\text{for all }u_n'\in X_n.
\end{align}
\end{subequations}
Of course it has to be ensured that $S$ and $S_n$ are well defined continuous
operators (for sufficiently large $n$) through
Equation~\eqref{eq:SolutionOperatorB}. The spectra of~\eqref{eq:ResonanceProblemVariational}
and~\eqref{eq:ResonanceProblemPMLGalerkin} are connected to the spectra of $S$
and $S_n$ respectively by the transformation
\begin{align}
\spm\mapsto \frac{1}{\spm^2-1}.
\end{align}
If $S$ is a compact operator it can be deduced that $S_n$ converges to $S$ in
operator norm which yields spectral convergence~\cite{BabuskaOsborn:91}.
However, the essential spectrum of $S$ equals $\{\frac{1}{z^2-1}\colon z\in\setC, \Re(izd_0)=0\}$
\cite[Theorem~4.6]{Halla:19Diss}.
Since the spectrum of a compact operator is discrete we deduce
that $S$ is not compact. Thus the standard theory~\cite{BabuskaOsborn:91,Boffi:10}
does not apply.

Differential operators with non-compact resolvent $S$ occur e.g.\ in
electromagnetism where sufficient conditions on the Galerkin spaces to ensure
spectral convergence have been obtained e.g.\
in~\cite{CaorsiFernandesRaffetto:00,Buffa:05}. The analysis therein is
based on~\cite{DesclouxNassifRappaz:78a}, which
state that
\begin{align}
\|S-S_n\|_n:=\sup_{u_n\in X_n\setminus\{0\}} \|(S-S_n)u_n\|_X/\|u_n\|_X \to 0 \quad\text{as}\quad n\to\infty
\end{align}
is sufficient to ensure spectral convergence. See also the very comprehensive
works \cite{ArnoldFalkWinther:10,ChristiansenWinther:13}.
However, in the previous references the essential spectrum consists only of one isolated eigenvalue with infinite dimensional eigenspace whereas in our case it consists of a continuum. Thus the techniques of~\cite{DesclouxNassifRappaz:78a,ArnoldFalkWinther:10,ChristiansenWinther:13} cannot be applied for our analysis. Roughly speaking we cannot hope to approximate an operator with a non-discrete essential spectrum by operators with discrete spectrum in a uniform way. All we can hope for is that we obtain locally (with respect to the spectral parameter) converging approximations.

Indeed local analysis techniques are the core of the holomorphic Fredholm
operator approximation theory \cite{Karma:96a,Karma:96b}.
We will recall the results of \cite{Halla:19Tcomp} in Section~\ref{sec:framework} and subsequently explain how to fit our eigenvalue problem at hand and its approximations into this theory.

\section{Analytical framework}\label{sec:framework}

We introduce some common notation and recall the framework of weakly $T(\cdot)$-coercive
operator functions and $T(\cdot)$-compatible Galerkin approximations \cite{Halla:19Tcomp}
in Subsection~\ref{subsec:Tframework}.
Subsequently in Subsection~\ref{subsec:Tconstruction} we construct a suitable $T(\cdot)$-operator
function for $a(\cdot;\cdot,\cdot)$ as defined in \eqref{eq:X-sprX-a}.

\subsection{T-analysis framework}\label{subsec:Tframework}

For generic Banach spaces $(X, \|\cdot\|_X)$, $(Y, \|\cdot\|_Y)$ denote $L(X,Y)$
the space of all bounded linear operators from $X$ to $Y$ with operator norm
$\|A\|_{L(X,Y)}:=\sup_{u\in X\setminus\{0\}} \|Au\|_Y/\|u\|_X$, $A\in L(X,Y)$.
We further set $L(X):=L(X,X)$. For generic Hilbert spaces
$(X, \spl\cdot,\cdot\spr_X)$, $(Y, \spl\cdot,\cdot\spr_Y)$ and $A\in L(X,Y)$ we denote $A^*\in L(Y,X)$ its adjoint
operator defined through $\spl u,A^*u' \spr_X=\spl Au,u'\spr_Y$ for all $u\in X,u'\in Y$.
We call an operator $A\in L(X)$ is coercive if $\inf_{u\in X\setminus\{0\}}
|\spl Au,u\spr_X|/\|u\|^2_X$ $>0$. We say that $A\in L(X)$ is weakly coercive, if there
exists a compact operator $K\in L(X)$ so that $A+K$ is coercive. For bijective $T\in L(X)$
we say that $A$ is (weakly) $T$-coercive, if $T^*A$ is (weakly) coercive.
Let $\Lambda\subset\setC$ be open and connected 
and consider operator functions $A(\cdot), T(\cdot)\colon \Lambda\to L(X)$
so that $T(\lambda)$ is bijective for all $\lambda\in\Lambda$. We call $A(\cdot)$ (weakly) ($T(\cdot)$-)coercive
if $A(\lambda)$ is (weakly) ($T(\lambda)$-)coercive for all $\lambda\in\Lambda$.
We denote the spectrum of $A(\cdot)$ as
$\sigma\big(A(\cdot)\big):=\{\lambda\in\Lambda\colon A(\lambda)\text{ is not bijective}\}$ and
the resolvent set as $\rho\big(A(\cdot)\big):=\Lambda\setminus\sigma\big(A(\cdot)\big)$.
For a closed subspace $X_n\subset X$ denote $P_n \in L(X,X_n)$ the orthogonal projection.
Consider $A\in L(X)$ to be weakly $T$-coercive. For a sequence $(X_n)_{n\in\setN}$ of closed subspaces
$X_n\subset X$ with $\lim_{n\in\setN}\|u-P_nu\|_X=0$ for each $u\in X$,
we say that the Galerkin approximation $P_nA|_{X_n}\in L(X_n)$ is $T$-compatible, if each $P_nA|_{X_n}$ is
Fredholm with index zero and there exists a sequence of Fredholm index zero operators
$(T_n)_{n\in\setN}, T_n\in L(X_n)$ so that
\begin{align}\label{eq:discretenorm}
\|T-T_n\|_n:=\sup_{u_n\in X_n\setminus\{0\}} \|(T-T_n)u_n\|_X /\|u_n\|_X
\end{align}
tends to zero as $n\to\infty$.
Let $A(\cdot)\colon \Lambda\to L(X)$ be weakly $T(\cdot)$-coercive. We say that the Galerkin approximation
$P_nA(\cdot)|_{X_n}\colon\Lambda\to L(X_n)$ is $T(\cdot)$-compatible, if $P_nA(\lambda)|_{X_n}\in L(X_n)$ is
$T(\lambda)$-compatible for each $\lambda\in\Lambda$.

We recall from \cite[Corollary~2.8]{Halla:19Tcomp}:
Let $A(\cdot)\colon\Lambda\to L(X)$ be a weakly $T(\cdot)$-coercive holomorphic operator function with non-empty resolvent set and $A_n(\cdot)\colon\Lambda\to L(X_n)$ be a $T(\cdot)$-compatible Galerkin approximation. Then
\begin{enumerate}[i)]
 \item\label{item:SP-a} For every eigenvalue $\spml_0$ of $A(\cdot)$ exists a sequence $(\spml_n)_{n\in\setN}$
 converging to $\spml_0$ with $\spml_n$ being an eigenvalue of $A_n(\cdot)$ for almost all $n\in\setN$.
 \item\label{item:SP-b} Let $(\spml_n, u_n)_{n\in\setN}$ be a sequence of normalized eigenpairs of $A_n(\cdot)$, i.e.\
 $A_n(\spml_n)u_n$ $=0$
 and $\|u_n\|_X=1$, so that $\spml_n\to \spml_0\in\Lambda$. Then
 \begin{enumerate}[a)]
  \item $\spml_0$ is an eigenvalue of $A(\cdot)$,
  \item $(u_n)_{n\in\setN}$ is a compact sequence and its cluster points are normalized eigenelements of $A(\spml_0)$.
 \end{enumerate}
 \item\label{item:SP-c} For every compact $\tilde\Lambda\subset\rho(A)$ the sequence $(A_n(\cdot))_{n\in\setN}$ is
 stable on  $\tilde\Lambda$, i.e.\ there exist $n_0\in\setN$ and $c>0$ such that $\|A_n(\spml)^{-1}\|_{L(X_n)}\leq c$ for
 all  $n>n_0$ and all $\spml\in\tilde\Lambda$.
 \item\label{item:SP-d} \label{item:Stability}For every compact $\tilde\Lambda\subset\Lambda$ with rectifiable boundary
 $\partial\tilde\Lambda\subset\rho\big(A(\cdot)\big)$ exists an index $n_0\in\setN$ such that
 \begin{align*}
 \dim G(A(\cdot),\spml_0) = \sum_{\spml_n\in\sigma\left(A_n(\cdot)\right)\cap\tilde\Lambda}
 \dim G(A_n(\cdot),\spml_n).
 \end{align*}
 for all $n>n_0$, whereby $G(B(\cdot),\spml)$ denotes the generalized eigenspace of an operator function
 $B(\cdot)$ at $\spml\in\Lambda$.
\end{enumerate}
Let $\tilde\Lambda\subset\Lambda$ be a compact set with rectifiable boundary
$\partial\tilde\Lambda\subset\rho\big(A(\cdot)\big)$, $\tilde\Lambda\cap\sigma\big(A(\cdot)\big)=\{\spml_0\}$ and
\begin{align}\label{eq:deltan}
\begin{split}
\delta_n&:=\max_{\substack{u_0\in G(A(\cdot),\spml_0)\\\|u_0\|_X\leq1}} \, \inf_{u_n\in X_n} \|u_0-u_n\|_X,\\
\delta_n^*&:=\max_{\substack{u_0\in G(A^*(\ol{\cdot}),\spml_0)\\\|u_0\|_X\leq1}} \, \inf_{u_n\in X_n} \|u_0-u_n\|_X,
\end{split}
\end{align}
whereby $\overline{\spml_0}$ denotes the complex conjugate of $\spml_0$ and $A^*(\cdot)$
the adjoint operator function of $A(\cdot)$ defined by $A^*(\spml):=A(\spml)^*$ for each $\spml\in\Lambda$.
Then there exist $n\in\setN$ and $c>0$ such that for all $n>n_0$
\begin{enumerate}[i)]
\setcounter{enumi}{4}
  \item\label{item:SP-e}
  \begin{align*}
  |\spml_0-\spml_n|\leq c(\delta_n\delta_n^*)^{1/\varkappa\left(A(\cdot),\spml_0\right)}
  \end{align*}
  for all $\spml_n\in\sigma\big(A_n(\cdot)\big)\cap\tilde\Lambda$, whereby
  $\varkappa\left(A(\cdot),\spml_0\right)$ denotes the maximal length of a Jordan
  chain of $A(\cdot)$ at the eigenvalue $\spml_0$,
  \item\label{item:SP-f}
  \begin{align*}
  |\spml_0-\spml_n^\mathrm{mean}|\leq c\delta_n\delta_n^*
  \end{align*}
  whereby $\spml_n^\mathrm{mean}$ is the weighted mean of all the eigenvalues of $A_n(\cdot)$ in $\tilde\Lambda$
  \begin{align*}
  \spml_n^\mathrm{mean}:=\sum_{\spml\in\sigma\left(A_n(\cdot)\right)\cap\tilde\Lambda}\spml\,
  \frac{\dim G(A_n(\cdot),\spml)}{\dim G(A(\cdot),\spml_0)},
  \end{align*}
  \item\label{item:SP-g}
  \begin{align*}
  \begin{split}
  \inf_{u_0\in\ker A(\spml_0)} \|u_n-u_0\|_X &\leq c \Big(|\spml_n-\spml_0|+
  \max_{\substack{u'_0\in\ker A(\spml_0)\\\|u_0'\|_X\leq1}} \inf_{u'_n\in X_n} \|u'_0-u'_n\|_X\Big)\\
  &\leq c\big(c(\delta_n\delta_n^*)^{1/\varkappa\left(A(\cdot),\spml_0\right)} + \delta_n\big)
  \end{split}
  \end{align*}
  for all $\spml_n\in\sigma\big(A_n(\cdot)\big)\cap\tilde\Lambda$ and all $u_n\in \ker A_n(\spml_n)$ with $\|u_n\|_X=1$.
\end{enumerate}

For the forthcoming analysis it will be more suitable to work with operators
instead of sesquilinearforms. Thus for a bounded $\spml$-dependent sesquilinearform
$a(\spml;\cdot,\cdot)\colon X\times X\to\setC$ we associate with the Riesz
representation theorem a $\spml$-dependent operator $A(\spml)\in L(X)$ defined through
\begin{align}\label{eq:Ak}
\langle A(\spml)u,u' \rangle_X = a(\spml;u,u') \qquad\text{for all}\quad u,u'\in X.
\end{align}
Eigenvalue Problem~\eqref{eq:ResonanceProblemVariational} can now be expressed as
\begin{align}\label{eq:ResonanceProblemOperator}
\text{find}\quad (\spml,u)\in\setC\times X\setminus\{0\}\quad\text{so that}\quad A(\spml)u=0.
\end{align}
If $X$ is approximated by a closed subspace $X_n$ we can associate a
$\spml$-dependent operator $A_n(\spml)\in L(X_n)$ through
\begin{align}\label{eq:Akn}
\langle A_n(\spml)u_n,u'_n \rangle_X = a(\spml;u_n,u'_n) \qquad\text{for all}\quad u_n,u'_n\in X_n.
\end{align}
Similarly~\eqref{eq:ResonanceProblemPMLGalerkin} can now be expressed as
\begin{align}\label{eq:ResonanceProblemOperatorGalerkin}
\text{find}\quad (\spml,u_n)\in\setC\times X_n\setminus\{0\}\quad\text{so that}\quad A_n(\spml)u_n=0.
\end{align}
The operators $A(\spml)$ and $A_n(\spml)$ are related through $A_n(\spml)=P_nA(\spml)|_{X_n}$.

\subsection{Construction of \texorpdfstring{$T(\cdot)$}{T(.)}}\label{subsec:Tconstruction}

Now we discuss the construction of a suitable $T(\cdot)$-operator function for the operator function
$A(\cdot)$ associated to the sesquilinear $a(\cdot;\cdot,\cdot)$ from \eqref{eq:X-sprX-a}.
Fortunately $T(\spm)$ can be realized as a simple multiplication operator.
For specific profiles of the Kind~\eqref{eq:AlphaSmooth} this was already exploited implicitly in~\cite{BramblePasciak:07}
from wherein the ansatz is taken and extended.
For our forthcoming analysis we additionally require the following assumption.

\begin{assumption}\label{ass:TildeAlpha2}
Let $\tilde\alpha$ and $r^*_1$ be as is Assumption~\ref{ass:TildeAlpha} and $\tilde d$, $d$ be as
in~\eqref{eq:ComplexScalingQuantities}. Let
\begin{enumerate}
 \item\label{it:limtdd} $\displaystyle\lim_{r\to+\infty} \tilde d(r)|d(r)|/\big(|\tilde d(r)|d(r)\big)=1$,
 \item\label{it:limdabsd} $\displaystyle\lim_{r\to+\infty} \big(\partial_r(\tilde d/|\tilde d|)\big)(r)
 =\lim_{r\to\infty} \big(\partial_r(d/|d|)\big)(r)=0$.
\end{enumerate}
\end{assumption}
Assumption~\ref{ass:TildeAlpha2}.\ref{it:limtdd} is necessary for Lemma~\ref{lem:arg} which will yield the essential
argument to prove the ``coercivity part'' in~Theorem~\ref{thm:AwTc}. Assumption~\ref{ass:TildeAlpha2}.\ref{it:limdabsd}
on the other hand will be necessary to prove the ``compactness part'' in~Theorem~\ref{thm:AwTc}.

It can easily be seen that any $\tilde\alpha$ of the Kind~\eqref{eq:AlphaAffin},
\eqref{eq:AlphaPower} and \eqref{eq:AlphaSmooth} suffices Assumption~\ref{ass:TildeAlpha2}.
In general, any reasonable profile function that comes to our mind suffices Assumption~\ref{ass:TildeAlpha2}.

Next we introduce two lemmata which will be essential for our analysis. Let
\begin{align}\label{eq:DefArg}
\arg z \colon \setC\setminus\{0\}\to [-\pi,\pi), \quad z=|z|\exp(i\arg z).
\end{align}
\begin{lemma}\label{lem:arg}
Let Assumptions~\ref{ass:TildeAlpha} and \ref{ass:TildeAlpha2} hold. Then there
exists $\tau\in(0,\pi/2)$ so that $\arg \big(d(r)/\tilde d(r)\big)\in [0,\tau]$
for all $r>r_1^*$.
\end{lemma}
\begin{proof}
Let $r>r_1^*$. Due $\tilde\alpha(r)\geq0$, the definition of $\tilde d, d$ and
Assumption~\ref{ass:TildeAlpha}.\ref{it:talphamonoton} it holds
$\arg \tilde d(r) \leq \arg d(r)$. Since $\arg \big(d(r)/\tilde d(r)\big) =
\arg d(r) -\arg \tilde d(r)$ it follows $\arg \big(d(r)/\tilde d(r)\big) \in [0,\pi/2)$.
Due to Assumption~\ref{ass:TildeAlpha}.\ref{it:talphaC2} $\tilde d/d$ is continuous.
Together with Assumption~\ref{ass:TildeAlpha2}.\ref{it:limtdd} it follows
$\sup_{r>r_1^*} \arg \big(d(r)/\tilde d(r)\big)<\pi/2$. Hence the claim is proven.
\end{proof}

\begin{lemma}\label{lem:WeigthedL2Compactness}
Let $\eta_1\colon\Omega\to\setC$ be measurable so that $\eta_1|_{\Omega\cap B_n}\in L^\infty(\Omega\cap B_n)$
for all $n\in\setN$. Let $Y\subset L^2(\Omega)$ be a Hilbert space so that $\|\eta_1 u\|_{L^2(\Omega)}\leq C\|u\|_Y$
for a constant $C>0$ and all $u\in Y$ and so that the embedding and restriction operator
$K_n\colon Y\to L^2(\Omega\cap B_n)\colon u\mapsto u|_{\Omega\cap B_n}$ is compact for each $n\in\setN$.
Let $\eta_2\in L^\infty(\Omega)$ be so that $\lim_{r\to\infty}\|\eta_2\|_{L^\infty(B_r^c)}=0$. Then the
multiplication and embedding operator $K_{\eta_1\eta_2}\colon Y\to L^2(\Omega)\colon u\mapsto \eta_1\eta_2 u$ is compact.
\end{lemma}
\begin{proof}
Consider a sequence $(u_n)_{n\in\setN}$ with $u_n\in Y, \|u_n\|_Y\leq1$. We
construct a Cauchy subsequence as follows. We choose a subsequence $N_1\colon\setN\to\setN$
so that $(K_1u_{N_1(n)})_{n\in\setN}$ converges. Iteratively for $m\in\setN$ we
choose subsequences $N_m\colon\setN\to\setN$ so that $(K_mu_{N_m(n)})_{n\in\setN}$
converges. Via diagonalization we construct a subsequence $N(n):=N_n(n)$. Let
$\epsilon>0$ and $n_1>0$ be so that $\|\eta_2\|_{L^\infty(B_{n_1}^c)}<\epsilon/(4C)$.
Let $n_2>0$ be so that
$\|K_{n_1}(u_{N(n)}-u_{N(n')})\|_{L^2(\Omega\cap B_{n_1})}<\epsilon/(2\|\eta_1\eta_2\|_{L^\infty(\Omega\cap B_{n_1})})$
for all $n,n'>n_2$. It follows
\begin{align*}
\|\eta_1\eta_2 (u_{N(n)}-u_{N(n')})\|_{L^2(\Omega)}
&\leq \|\eta_1\eta_2\|_{L^\infty(\Omega)} \|u_{N(n)}-u_{N(n')}\|_{L^2(\Omega\cap B_{n_1})}
\\&+2\|\eta_2\|_{L^\infty(B_{n_1}^c)}
< \epsilon.
\end{align*}
Hence the claim is proven.
\end{proof}

Our analysis will further require the following functions.
\begin{subequations}\label{eq:dcont}
\begin{align}
\label{eq:HatAlpha}
\hat\alpha(r)&:=\left\{\begin{array}{rl}\lim_{\rho\to r_1^*+}\alpha(\rho)&\text{for }0\leq r\leq r_1^*,\\
\alpha(r)&\text{for }r>r_1^*,\end{array}\right.\\
\label{eq:Hatd}
\hat d(r)&:=1+i\hat\alpha(r), \quad r\geq0.
\end{align}
\end{subequations}
Again, we adopt the overloaded notation $f(x):=f(|x|), x\in\Omega$ for $f=\hat\alpha,\hat d$.

\begin{lemma}\label{lem:TisIsomorphism}
Let Assumptions~\ref{ass:TildeAlpha} and \ref{ass:TildeAlpha2} hold.
For all $\spm\in\setC\setminus\{0\}$ and all $u\in X$ let
\begin{align}\label{eq:DefTHH}
T(\spm)u:=\left\{\begin{array}{ll}
\ol{\frac{|\hat d|}{\hat d}} u &\text{for } \arg(-\spm^2d_0^2)\in [-\pi,0), \vspace{3mm}\\
\ol{\frac{\hat d}{\tilde d^2}\frac{|\tilde d|^2}{|\hat d|}} u &\text{for }\arg (-\spm^2d_0)\in [0,\pi).
\end{array}\right.
\end{align}
Then $T(\spm)\in L(X)$ is bijective for all $\spm\in\setC\setminus\{0\}$.
\end{lemma}
\begin{proof}
For any $\eta\in W^{1,\infty}(\Omega)$ and $u\in X$ it holds
\begin{align*}
\|\eta u\|_X^2
&= \langle (|\tilde d^2d^{-1}|\Px+|d|(\I-\Px))(\eta\nabla u+u\nabla\eta),\eta\nabla u+u\nabla\eta\rangle_{L^2(\Omega)}
\\&+\langle |\tilde d^2d|\eta u,\eta u \rangle_{L^2(\Omega)}
\leq 3\|\eta\|_{W^{1,\infty}(\Omega)}^2 \|u\|_X^2.
\end{align*}
Thus multiplication with $\eta$ is bounded from $X\to X$.
If $|\eta|=1$ it follows $1/\eta \in W^{1,\infty}(\Omega)$ as well.
Hence the inverse of multiplication with $\eta$, which is multiplication
with $1/\eta$, is bounded from $X\to X$ as well.

Let $\eta=\ol{\frac{|\hat d|}{\hat d}}$ or $\eta=\ol{\frac{\hat d}{\tilde d^2}\frac{|\tilde d|^2}{|\hat d|}}$.
It follows $|\eta|=1$. Due to the definition of $\hat d$~\eqref{eq:dcont} and Assumption~\ref{ass:TildeAlpha},
$\eta$ is weakly differentiable. Due to Assumption~\ref{ass:TildeAlpha}.\ref{it:talphaC2} and
Assumption~\ref{ass:TildeAlpha2}.\ref{it:limdabsd} it follows $\nabla \eta\in L^\infty(\Omega)$ and hence
$\eta\in W^{1,\infty}(\Omega)$. Thus the claim is proven.
\end{proof}

\begin{theorem}\label{thm:AwTc}
Let Assumptions~\ref{ass:TildeAlpha} and \ref{ass:TildeAlpha2} hold.
Let $a(\cdot;\cdot,\cdot)$ and $X$ be as in~\eqref{eq:X-sprX-a}, $A(\cdot)$ be as
in~\eqref{eq:Ak}, $T(\cdot)$ be as in~\eqref{eq:DefTHH}, $d_0$ be as in~\eqref{eq:dzDef} and
\begin{align}\label{eq:Lambdad0}
\Lambda_{d0}:=\{z\in\setC\colon \Re(izd_0)\neq0\}.
\end{align}
Then $A(\cdot)\colon\Lambda_{d_0}\to L(X)$ is weakly $T(\cdot)$-coercive.
\end{theorem}
\begin{proof}
We consider the two cases $\arg(-\spm^2d_0^2)\in(-\pi,0)$ and $\arg(-\spm^2d_0^2)\in[0,\pi)$ separately. We split
the sesquilinear form $a(\spm;\cdot,T(\spm)\cdot)$ into a coercive part $a_1(\cdot,\cdot)$ and a compact part
$a_2(\cdot,\cdot)$.\\
\textit{First case $\spm\in\Lambda_{d_0}$ with $\arg(-\spm^2d_0^2)\in(-\pi,0)$:}\quad
A direct computation yields
\begin{align*}
a(\spm;u,T(\spm)u')=a_1(u,u')+a_2(u,u')
\end{align*}
with
\begin{align*}
a_1(u,u') &:= \bigg\langle \bigg(\frac{\tilde d^2|\hat d|}{d\hat d}\Px+\frac{d|\hat d|}{\hat d}(\I-\Px)\bigg)\nabla u,
\nabla u'\bigg\rangle_{L^2(\Omega)}-\spm^2d_0^2\langle |\tilde d^2d|u,u'\rangle_{L^2(\Omega)},\\
a_2(u,u') &:= \bigg\langle \frac{\tilde d^2}{d}\partial_r u,u'\partial_r\ol{\left(\frac{|\hat d|}{\hat d}\right)}\bigg\rangle_{L^2(\Omega)}
-\spm^2\bigg\langle\bigg(\frac{\tilde d^2d|\hat d|}{|\tilde d^2d|\hat d}-d_0^2\bigg)|\tilde d^2d|u,u'\bigg\rangle_{L^2(\Omega)}.
\end{align*}
Recall that $\hat d(r)=d(r)$ for $r>r_1^*$ and $\hat d(r)=1+i\lim_{r\to r_1^*+}\alpha(r)$
for $r\leq r_1^*$. Due to Assumptions \ref{ass:TildeAlpha}.\ref{it:talphamonoton} and \ref{ass:TildeAlpha}.\ref{it:talphaC2}
it holds $\arg \hat d(r_1^*)\in [0,\pi/2)$. Let $\tau\in(0,\pi/2)$ be as in Lemma~\ref{lem:arg} and
\begin{align*}
\tau_1:=\min \{ -2\tau, -\arg \hat d(r_1^*), \arg(-\spm^2d_0^2) \}.
\end{align*}
It follows that $\tau_1\in(-\pi,0)$ and
\begin{align*}
\arg \left(\frac{\tilde d^2|\hat d|}{d\hat d}\right)(r), \, \arg\left(\frac{d|\hat d|}{\hat d}\right)(r),
\, \arg(-\spm^2d_0^2) \in [\tau_1,0]
\end{align*}
for all $r\geq0$. Thus $\Re(ie^{-i(\pi+\tau_1)/2}a_1(u,u)) \geq \cos(\tau_1/2)\min\{1,|\spm^2|\} \|u\|_X^2$
for all $u\in X$, i.e.\ $a_1(\cdot,\cdot)$ is coercive.
Further $a_2(u,u')=\langle (K_1^*L_1-\spm^2K_2^*L_2)u,u'\rangle_X$ with bounded operators
\begin{align*}
L_1&\colon X\to L^2(\Omega)\colon u\mapsto \frac{\tilde d}{d^{1/2}} \partial_r u,\\
K_1&\colon X\to L^2(\Omega)\colon u\mapsto \left(\ol{\partial_r\left(\frac{|\hat d|}{\hat d}\right)}\right) \ol{\frac{\tilde d}{d^{1/2}}}u,\\
L_2&\colon X\to L^2(\Omega)\colon u\mapsto |\tilde dd^{1/2}| u,\\
K_2&\colon X\to L^2(\Omega)\colon u\mapsto \left(\frac{\tilde d^2d|\hat d|}{|\tilde d^2d|\hat d}-d_0^2\right) |\tilde dd^{1/2}|u.
\end{align*}
From the definitions of $d_0$ and $\hat d$ it follows $\left(\frac{\tilde d^2d|\hat d|}{|\tilde d^2d|\hat d}-d_0^2\right)(r)\to0$
as $r\to+\infty$. From Assumption~\ref{ass:TildeAlpha2}.\ref{it:limdabsd} follows
$\left(\ol{\partial_r\left(\frac{|\hat d|}{\hat d}\right)}\right)(r)\to0$ as $r\to+\infty$.
Lemma~\ref{lem:WeigthedL2Compactness}, $1/|d|\leq1$ and the compact Sobolev embedding $H^1(D)\to L^2(D)$
for bounded Lipschitz domains $D$ yield that $K_1$ and $K_2$ are compact. Hence $A_2$ ($\spl A_2u,u'\spr_X=a_2(u,u')$)
is compact too.\\
\textit{Second case $\spm\in\Lambda_{d_0}$ with $\arg(-\spm^2d_0^2)\in[0,\pi)$:}\quad
A direct computation yields $a(\spm;u,T(\spm)u')=a_1(u,u')+a_2(u,u')$ with
\begin{align*}
a_1(u,u') &:= \bigg\langle \bigg(\frac{\hat d|\tilde d^2|}{d|\hat d|}\Px+\frac{d\hat d|\tilde d^2|}{\tilde d^2|\hat d|}(\I-\Px)\bigg)\nabla u,
\nabla u'\bigg\rangle_{L^2(\Omega)}-\spm^2d_0^2\langle |\tilde d^2d|u,u'\rangle_{L^2(\Omega)},\\
a_2(u,u') &:= \bigg\langle \frac{\tilde d^2}{d}\partial_r u,u'\partial_r\ol{\left(\frac{\hat d|\tilde d^2|}{\tilde d^2|\hat d|}\right)}\bigg\rangle_{L^2(\Omega)}
-\spm^2\bigg\langle\bigg(\frac{\hat d|\tilde d^2|\tilde d^2d}{\tilde d^2|\hat d||\tilde d^2d|}-d_0^2\bigg)|\tilde d^2d|u,u'\bigg\rangle_{L^2(\Omega)}.
\end{align*}
As in the previous case we find that
\begin{align*}
\arg \left(\frac{\hat d|\tilde d^2|}{d|\hat d|}\right)(r), \, \arg\left(\frac{d\hat d|\tilde d^2|}{\tilde d^2|\hat d|}\right)(r),
\, \arg(-\spm^2d_0^2) \in [0,\tau_1]
\end{align*}
for all $r\geq0$ with $\tau_1:=\max \{ 2\tau, \arg \hat d(r_1^*), \arg(-\spm^2d_0^2) \}\in[0,\pi)$.
It follows
\begin{align*}
\Re(-ie^{i(\pi-\tau_2)/2}a_1(u,u)) \geq \cos(\tau_1/2) \min\{1,|\spm^2|\} \|u\|_X^2
\end{align*}
for all $u\in X$, i.e.\ $a_1(\cdot,\cdot)$ is coercive.
Further $a_2(u,u')=\langle (K_1^*L_1-\spm^2K_2^*L_2)u,u'\rangle_X$ with bounded operators
\begin{align*}
L_1&\colon X\to L^2(\Omega)\colon u\mapsto \frac{\tilde d}{d^{1/2}} \partial_r u,\\
K_1&\colon X\to L^2(\Omega)\colon u\mapsto \left(\ol{\partial_r\left(\frac{\hat d|\tilde d^2|}{\tilde d^2|\hat d|}\right)}\right) \ol{\frac{\tilde d}{d^{1/2}}}u,\\
L_2&\colon X\to L^2(\Omega)\colon u\mapsto |\tilde dd^{1/2}| u,\\
K_2&\colon X\to L^2(\Omega)\colon u\mapsto \left(\frac{\hat d|\tilde d^2|\tilde d^2d}{\tilde d^2|\hat d||\tilde d^2d|}-d_0^2\right) |\tilde dd^{1/2}|u.
\end{align*}
From the definitions of $d_0$, $\hat d$ and Assumption~\ref{ass:TildeAlpha2}.\ref{it:limtdd}
follows $\left(\frac{\hat d|\tilde d^2|\tilde d^2d}{\tilde d^2|\hat d||\tilde d^2d|}-d_0^2\right)(r)\to0$
as $r\to+\infty$. From Assumption~\ref{ass:TildeAlpha2}.\ref{it:limdabsd} it follows
$\left(\ol{\partial_r\left(\frac{\hat d|\tilde d^2|}{\tilde d^2|\hat d|}\right)}\right)(r)\to0$ as $r\to+\infty$.
Again, Lemma~\ref{lem:WeigthedL2Compactness}, $1/|d|\leq1$ and the compact Sobolev embedding $H^1(D)\to L^2(D)$ for
bounded Lipschitz domains $D$ yield that $K_1$ and $K_2$ are compact. Hence $A_2$ ($\spl A_2u,u'\spr_X=a_2(u,u')$)
is compact too.
\end{proof}

If we demand in addition to Assumptions~\ref{ass:TildeAlpha}, \ref{ass:TildeAlpha2} that $|d|/|\tilde d|$ is
bounded, then it can be proven \cite[Theorem 4.6]{Halla:19Diss} that Theorem~\ref{thm:AwTc} is sharp, i.e.\
the essential spectrum of $A(\cdot)$ equals $\setC\setminus\Lambda_{d_0}=\{\omega\in\setC\colon\Re(-i\omega d_0)=0\}$.

It is less intuitive why we need to employ the multiplication operator $T(\spm)$.
The matrix of the principle part of $a(\spm;\cdot,\cdot)$ is $\tilde d^2d^{-1}\Px+d(\I-\Px)$.
The coefficients are bounded away from zero and only take values in the closed salient sector
spanned by $(1+i\alpha(r))^{\pm1}$. However, as the domain is unbounded the (asymptotic) complex sign
of the $L^2$-term $-\spm^2d_0^3$ also has to be taken into account. Although there is no way to estimate
$1+i\alpha(r)$ in terms of $d_0$ without further assumptions on $\tilde\alpha$. Nonetheless the asymptotic complex sign
of the matrix coefficients is $d_0$. Thus it is meaningful to suitably rotate the complex sign of the principle part
especially in the preasymptotic regime of the coefficients. The rotation for the $L^2$-term in the preasymptotic regime
can be neglected as $L^2$-integrals on bounded sets lead to compact operators. In~\cite{BramblePasciak:07} it was noted that 
a rotation by $d^{-1}$ yields the desired properties.
Be aware that for different dimensions and different equations other rotations
are necessary. A choice leading to coefficients $1$ and $\tilde d^2/d^2$ or $1$
and $d^2/\tilde d^2$ (depending on the complex sign of $-\spm^2d_0^2$)
in the principle part of the equation usually does the job.

\section{Subsequent approximation}\label{sec:appr_subseq}

We consider a sequence of finite subdomains $(\Omega_n)_{n\in\setN}$ which suffices
Assumption~\ref{ass:OmegaN}, corresponding subspaces $X_n$ defined by~\eqref{eq:Xn}
and corresponding operator functions $A_n(\cdot)$ defined by~\eqref{eq:Akn}. We
investigate the approximation of $A(\cdot)$ by $A_n(\cdot)$.

\begin{lemma}\label{lem:XnTinvariant}
Let Assumptions~\ref{ass:TildeAlpha}, \ref{ass:TildeAlpha2} and \ref{ass:OmegaN} hold.
Let $X_n$ be as in~\eqref{eq:Xn} and $T(\cdot)$ be as in~\eqref{eq:DefTHH}.
Then $X_n$ is $T(\spm)$-invariant and $T^{-1}(\spm)$-invariant for all $n\in\setN$, $\spm\in\setC\setminus\{0\}$,
i.e.\ $T(\spm)u_n, T^{-1}(\spm)u_n\in X_n$ for all $u_n\in X_n$, $n\in\setN$, $\spm\in\setC\setminus\{0\}$.
\end{lemma}
\begin{proof}
A multiplication operator does not increase the support of a function.
\end{proof}

\begin{theorem}[Spectral convergence]\label{thm:SpectralConvergenceHelmholtz}
Let Assumptions~\ref{ass:TildeAlpha} and \ref{ass:TildeAlpha2} hold.
Let $X$ and $a(\cdot;\cdot,\cdot)$ be as in~\eqref{eq:X-sprX-a},
$A(\cdot)$ be as in~\eqref{eq:Ak}, $T(\cdot)$ be as in~\eqref{eq:DefTHH} and
\begin{align}\label{eq:Lambdad0pm}
\Lambda_{d_0}^\pm:=\{z\in\setC\colon \pm\Re(izd_0)<0\}.
\end{align}
Let Assumption~\ref{ass:OmegaN} hold.
Let $X_n$ be as in~\eqref{eq:Xn} and $A_n(\cdot)$ be as in~\eqref{eq:Akn}.

Then $A(\cdot)\colon \Lambda_{d_0}^\pm\to L(X)$ is a weakly $T(\cdot)$-coercive
holomorphic Fredholm operator function with non-empty resolvent set $\rho\big(A(\cdot)\big)$
and $A_n(\cdot)\colon \Lambda_{d_0}^\pm\to L(X_n)$ is a $T(\cdot)$-compatible approximation,
i.e.\ the convergence results \ref{item:SP-a})-\ref{item:SP-g}) from Subsection~\ref{subsec:Tframework} hold.
\end{theorem}
\begin{proof}
Since $A(\spm)$ is a polynomial in $\spm$ it is holomorphic.
Due to Theorem~\ref{thm:AwTc} $A(\spm)$ is weakly $T(\spm)$-coercive.
Due to \cite[Theorem~2.16]{Halla:19Diss} and \cite[Theorems~2.8-2.9]{Halla:19Diss}
the resolvent sets $\rho\big(A(\cdot)\big)\cap\Lambda_{d_0}^\pm$ are non-empty.
Thus $A(\cdot)\colon \Lambda_{d_0}^\pm\to L(X)$ is a weakly $T(\cdot)$-coercive
holomorphic Fredholm operator function with non-empty resolvent set $\rho\big(A(\cdot)\big)$.
Due to Lemma~\ref{lem:XnTinvariant} $X_n$ is $T(\spm)$-invariant for all $n\in\setN$, $\spm\in\Lambda_{d_0}^\pm$.
Hence with $T_n(\spm):=T(\spm)|_{X_n}$ it hold $T_n(\spm), T^{-1}_n(\spm) \in L(X_n)$
and $\|T(\spm)-T_n(\spm)\|_n=0$ for all $n\in\setN$, $\spm\in\Lambda_{d_0}^\pm$ and consequently
$A_n(\spm)$ and $T_n(\spm)$ are Fredholm with index zero for all $\spm\in\Lambda_{d_0}^\pm$.
\end{proof}

Theorem~\ref{thm:SpectralConvergenceHelmholtz} yields convergence rates with respect to the
best approximation errors~\eqref{eq:deltan}.
To estimate these we introduce the next Lemma~\ref{lem:BestApproximation}.

\begin{lemma}\label{lem:BestApproximation}
Let Assumptions~\ref{ass:TildeAlpha}, \ref{ass:TildeAlpha2} and~\ref{ass:OmegaN} hold.
Let $X$ be as in~\eqref{eq:X-sprX-a} and $X_n$ be as in~\eqref{eq:Xn}.
Let $r_n>0$ be so that $\Omega\cap B_{r_n+1}\subset\Omega_n$. Then there
exists a constant $C>0$ independent of $n$ so that
\begin{align}
\inf_{u_n\in X_n} \|u-u_n\|_X \leq C \|u\|_{X(B_{r_n}^c)}
\end{align}
for all $u\in X$.
\end{lemma}
\begin{proof}
We choose $u_n(x):=\cofii(1+r_n-|x|)u(x)\in X_n$ with $\cofii$ as in~\eqref{eq:Smooth0to1} and compute
\begin{align*}
\|u-u_n\|_X=\|u-u_n\|_{X(\Omega\setminus B_{r_n})} \leq \|u\|_{X(\Omega\setminus B_{r_n})}+C\|u\|_{X(A_{r_n,r_n+1})}
\end{align*}
with a constant $C>0$ independent of $u$ and $n$.
\end{proof}

Due to \eqref{eq:EstXNormU} $\|u\|_{X(B_{r_n}^c)}$ can be estimated
to decay exponentially for eigenfunctions $u$, i.e.\ $A(\spm)u=0$. For generalized
eigenfunctions (also called root functions) $\|u\|_{X(B_{r_n}^c)}$ can also be
estimated to decay exponentially due to Lemma~6.1 of~\cite{KimPasciak:09}.
Together with Theorem~\ref{thm:SpectralConvergenceHelmholtz} and Lemma~\ref{lem:BestApproximation} this yields convenient error bounds of the form
\begin{align*}
\|u-u_n\|_X \lesssim e^{-Cr_n}
\end{align*}
for some constant $C>0$ with layer width $r_n$.

For solutions $(\spm,u)$ to $A(\spm)u=0$ the quantity of interest is actually only
$(\spm,u|_{\Omega\cap B_{r_1^*}})$ where as $u|_{B_{r_1^*}^c}$ could be called an
auxiliary variable. It is indeed possible to improve the error estimate obtained
by Theorem~\ref{thm:SpectralConvergenceHelmholtz} and Lemma~\ref{lem:BestApproximation}
for the eigenspaces if the error is only measured in $\|\cdot\|_{X(\Omega\cap B_{r_1^*})}$.
A hand waving explanation is that $A_n(\cdot)$ differs from $A(\cdot)$ only by a
distortion at $B_{r_2^*}^c$ and as ``the error propagates'' towards $\Omega\cap B_{r_1^*}$
``the error decays''. This argumentation can be made rigorous by a comparison of
the Dirichlet-to-Neumann operators generated by the complex scaling in the
(un)truncated domains. For details see e.g.\ \cite[Section~4.3]{HallaNannen:18}.\\

Now consider for a fixed index $n\in\setN$ a subsequent approximation of
\eqref{eq:ResonanceProblemPMLGalerkin} by the following.
We consider a sequence of subspaces $\big(X_n^{h(m)}\big)_{m\in\setN}$, $X_n^{h(m)}\subset X_n$, $m\in\setN$,
so that the orthogonal projections $P_n^{h(m)}\colon X_n\to X_n^{h(m)}$ converge point-wise to
the identity in $X_n$ and eigenvalue problem
\begin{align}\label{eq:ResonanceProblemPMLGalerkinFEM}
\begin{split}
\text{find }(\spm,u_{h(m)})\in\setC\times X_n^{h(m)}\setminus\{0\}\text{ so that}
\quad a(\spm;u_{h(m)},u_{h(m)}')=0\\\text{for all }u_{h(m)}'\in X_n^{h(m)}.
\end{split}
\end{align}
We note that restricted to $X_n$ the norm $\|\cdot\|_X$ is equivalent to
$\|\cdot\|_{H^1(\Omega)}$ and hence $X_n=\{u\in H^1_0(\Omega)\colon u=0 \text{ in }\Omega\setminus\Omega_n\}$.
It holds further that $A_n(\spm)\in L(X_n)$ is already weakly coercive.
Hence the approximation of \eqref{eq:ResonanceProblemPMLGalerkin} by \eqref{eq:ResonanceProblemPMLGalerkinFEM}
can already be performed with common techniques \cite{BabuskaOsborn:91}.\\

The profile function $\tilde\alpha$ limits the regularity of solutions. However,
to achieve optimal approximations rates of solutions by general finite element
spaces smooth solutions are necessary. Yet, if $\tilde\alpha$ is piece-wise smooth
optimal rates can be restored if the meshes of the finite element spaces are aligned
to the jumps in the derivatives of $\tilde\alpha$.
If this is not possible, e.g.\ because the finite element code is limited to
polytopial meshes, it is desirable to chose a globally smooth profile function.
Of course for finite element spaces with fixed maximal polynomial degree one can
construct $\tilde\alpha$ with appropriate smoothness as piece-wise polynomial.
However, in this case it seems more natural to us to construct $\tilde\alpha
\in C^\infty(\setR^+)$ in the first place, e.g.\ as in \eqref{eq:AlphaInfty}.

\section{Simultaneaous approximation}\label{sec:appr_simul}

In the previous section we considered a sequence of bounded subdomains
$(\Omega_n)_{n\in\setN}$ as in Assumption~\ref{ass:OmegaN}, an approximation
of \eqref{eq:ResonanceProblemVariational} by \eqref{eq:ResonanceProblemPMLGalerkin}
and subsequent a sequence of subspaces $\big(X_n^{h(m)}\big)_{m\in\setN}$, $X_n^{h(m)}\subset X_n$
and an approximation of \eqref{eq:ResonanceProblemPMLGalerkin} by \eqref{eq:ResonanceProblemPMLGalerkinFEM}.
The two key ingredients which allowed a pretty simple analysis were the $T(\cdot)$-invariance
of $X_n$ and the weak coercivity of $A_n(\cdot)$. This way we avoided to discuss the
issue of the non-$T(\cdot)$-invariance of $X_n^{h(m)}$ and the construction of an
appropriate $T_n^{h(m)}(\cdot)$ operator function.
Though this kind of analysis yields only limited results: It is left open if a sequence of approximations with simultaneaous increasing domains and decreasing mesh-width' could lead to erroneous results (e.g.\ failure of convergence, spectral pollution). Nevertheless, the only work known to us (from the huge amount of articles on PML-approximations) which adressed this important issue so far is \cite{HohageNannen:15}.

Thus in this section we consider a direct approximation of \eqref{eq:ResonanceProblemVariational} through non-$T(\cdot)$-invariant subspaces of $X$, e.g.\ the diagonal sequence $\big(X_n^{h(n)}\big)_{n\in\setN}$.
To conduct our analysis we introduce an operator function $T_\epsilon(\cdot)$ which is a slight modification of~\eqref{eq:DefTHH} in Lemmata~\ref{lem:EtaEps} and~\ref{lem:Teps}.
This new operator function has some favorable properties and is so that $A(\cdot)$ is still
weakly $T_\epsilon(\cdot)$-coercive. We consider finite dimensional Galerkin spaces $X^{h(m)}\subset X$
which suffice two Assumptions~\ref{ass:LocalApproximation} and~\ref{ass:ConstantFunctions}. 
In Theorem~\ref{thm:Tepsh} we prove that under such assumptions we can construct appropriate
operator functions $T_\epsilon^{h(m)}(\cdot)\colon\setC\to L(X^{h(m)})$ which converge to $T_\epsilon(\cdot)$
in discrete norm at each $\spm\in\setC$, i.e.\ the Approximation $\big(A^{h(m)}(\cdot)\colon\Lambda_{d_0}\to L(X^{h(m)})\big)_{m\in\setN}$
is $T_\epsilon(\cdot)$-compatible. A key ingredient for the analysis is a variant of
the discrete commutator property of Bertoluzza~\cite{Bertoluzza:99}.
Finally in Theorem~\ref{thm:SpectralConvergenceDirect} we formulate our convergence results.

\begin{lemma}\label{lem:EtaEps}
Let $r_1\in\setR$ and $r_2\in\setR\cup\{+\infty\}$ with $r_1<r_2$.
Let $\eta\colon [r_1,r_2)\to\setC$ be continuous so that $\lim_{r\to r_2-}\eta(r)=:\eta(r_2)$
exists in $\setC$.
Then for each $\epsilon>0$ exist $\eta_\epsilon\colon [r_1,r_2)\to\setC$ and $\hat r_1, \hat r_2\in (r_1,r_2)$
so that
\begin{enumerate}
 \item $\|\eta-\eta_\epsilon\|_{L^\infty(r_1,r_2)}<\epsilon$,
 \item $\eta_\epsilon(r)=\eta(r_1)$ for $r \leq \hat r_1$,
 \item $\eta_\epsilon(r)=\eta(r_2)$ for $r \geq \hat r_2$,
 \item $\eta_\epsilon$ is infinitely many times differentiable.
\end{enumerate}
\end{lemma}
\begin{proof}
Since $\eta$ is continuous and $\lim_{r\to r_2-}\eta(r)$ exists we can choose
$\check r_1, \check r_2\in (r_1,r_2)$ so that $\|\eta-\eta(r_1)\|_{L^\infty(r_1,\check r_1)}<\epsilon/2$
and $\|\eta-\eta(r_2)\|_{L^\infty(\check r_2,r_2)}<\epsilon/2$.
Since $C^\infty(r_1,r_2)$ is dense in $L^\infty(r_1,r_2)$ we can choose $\hat\eta\in C^\infty(r_1,r_2)$
with $\|\eta-\hat\eta\|_{L^\infty(r_1,r_2)}<\epsilon/2$. Let $\hat r_1\in (r_1,\check r_1)$,
$\hat r_2\in (\check r_2, r_2)$ and
\begin{align*}
\eta_\epsilon:=
\left\{\begin{array}{ll}
\eta(r_1), & r\leq \hat r_1,\\
\big(1-\cofii(\frac{r-\hat r_1}{\check r_1-\hat r_1})\big)\eta(r_1) + \cofii(\frac{r-\hat r_1}{\check r_1-\hat r_1})\hat\eta(r), & \hat r_1<r\leq \check r_1,\\
\hat\eta(r), & \check r_1 <r< \check r_2,\\
\big(1-\cofii(\frac{r-\check r_1}{\hat r_2-\check r_2})\big)\hat\eta(r) + \cofii(\frac{r-\check r_1}{\hat r_2-\check r_2})\eta(r_2), & \hat r_2<r\leq \check r_2,\\
\eta(r_2), & r\geq \hat r_2,
\end{array}\right.
\end{align*}
with $\cofii$ as in~\eqref{eq:Smooth0to1}. From the triangle inequality and $\chi_2(r)\in [0,1]$
for all $r\in\setR$ it follows $\|\eta-\eta_\epsilon\|_{L^\infty(r_1,r_2)}<\epsilon$.
By construction $\eta_\epsilon$ suffices also the last three criteria.
\end{proof}

\begin{lemma}\label{lem:Teps}
Let Assumptions~\ref{ass:TildeAlpha} and \ref{ass:TildeAlpha2} hold.
Let $X$ and $a(\cdot;\cdot,\cdot)$ be as in~\eqref{eq:X-sprX-a} and
$A(\cdot)$ be as in~\eqref{eq:Ak}. For $\epsilon>0$ and $\spm\in\setC\setminus\{0\}$ let
\begin{align}\label{eq:DefTeps}
T_\epsilon(\spm)u:=\eta_\epsilon u \qquad\text{with} \qquad
\eta:=\left\{\begin{array}{ll}
\ol{\frac{|\hat d|}{\hat d}} &\text{for } \arg(-\spm^2d_0^2)\in [-\pi,0), \vspace{3mm}\\
\ol{\frac{\hat d}{\tilde d^2}\frac{|\tilde d|^2}{|\hat d|}} &\text{for }\arg (-\spm^2d_0^2)\in [0,\pi)
\end{array}\right.
\end{align}
and $\eta_\epsilon|_{(r_1^*,+\infty)}$ as in Lemma~\ref{lem:EtaEps} with $r_1=r_1^*, r_2=+\infty$
and $\eta_\epsilon|_{[0,r_1^*]}:=\eta_\epsilon(r_1^*)$.

For each $\omega\in\setC\setminus\{0\}$ there exists $\epsilon_0(\omega)>0$ so that for each
$\epsilon\leq\epsilon_0(\omega)$, $T_\epsilon(\spm) \in L(X)$ is bijective and $A(\omega)\colon \Lambda_{d_0}\to L(X)$
is weakly $T_\epsilon(\omega)$-coercive.
\end{lemma}
\begin{proof}
$T_\epsilon(\spm) \in L(X)$ and its bijectivity can be proven for a sufficiently small $\epsilon$
as in the proof of Lemma~\ref{lem:TisIsomorphism}. Similarly the weak
$T_\epsilon(\omega)$-coercivity of $A(\omega)$ can be proven for a sufficiently small
$\epsilon$ as in the proof of Theorem~\ref{thm:AwTc}.
\end{proof}

\begin{assumption}\label{ass:LocalApproximation}
There exists a sequence $\big(h(n)\big)_{n\in\setN}\in (\setR^+)^\setN$ with $\lim_{n\in\setN} h(n)=0$.
There exist bounded linear projection operators $\Pi_{h(n)}\colon X\to X^{h(n)}, n\in\setN$
that act locally in the following sense: there exist constants $C_1, R^*>1$ so that for $n\in\setN$, $s \in \{1,2\}$,
$x_0\in\Omega$, if $B_{R^*h(n)}(x_0)\subset \Omega$, $u\in X$ and $u|_{B_{R^*h(n)}(x_0)} \in H^s(B_{R^*h(n)}(x_0))$, then
\begin{align}
\|u-\Pi_{h(n)} u\|_{H^1(B_{h(n)}(x_0))} \leq C_1 h(n)^{s-1} \|u\|_{H^s(B_{R^*h(n)}(x_0))}.
\end{align}
\end{assumption}

\begin{assumption}\label{ass:ConstantFunctions}
For any $D\subset\Omega$ which is compact in $\Omega$ exists $n_0>0$
so that for each $n\in\setN, n>n_0$ there exists $u_{D,n}\in X^{h(n)}$ with $u_{D,n}|_D=1$.
\end{assumption}

\begin{theorem}\label{thm:Tepsh}
Let Assumptions~\ref{ass:TildeAlpha} and \ref{ass:TildeAlpha2} hold.
Let $X$ be as in~\eqref{eq:X-sprX-a}, $\big(X^{h(n)}\big)_{n\in\setN}$ be sequence of
finite dimensional subspaces $X^{h(n)}\subset X$ so that the orthogonal projections
from $X$ onto $X^{h(n)}$ converge point-wise to the identity in $X$ and so that
Assumptions~\ref{ass:LocalApproximation} and~\ref{ass:ConstantFunctions} hold.
Let $\epsilon_0(\omega)$ be as in Lemma~\ref{lem:Teps}, $T_{\epsilon_0}(\omega):=T_{\epsilon_0(\omega)}(\omega)$
be as in~\eqref{eq:DefTeps} and $\|\cdot\|_n$ be as in~\eqref{eq:discretenorm}.
For $n\in\setN$ let $\Pi_{h(n)}$ be as in Assumptions~\ref{ass:LocalApproximation} and
\begin{align}
T_{\epsilon_0}^{h(n)}(\spm):=\Pi_{h(n)}T_{\epsilon_0}(\spm)|_{X^{h(n)}}
\end{align}
for $\spm\in\setC\setminus\{0\}$. Then $T_{\epsilon_0}^{h(n)}(\spm) \in L(X^{h(n)})$ is Fredholm with index zero and
\begin{align}
\lim_{n\in\setN} \|T_{\epsilon_0}(\spm)-T_{\epsilon_0}^{h(n)}(\spm)\|_n=0
\end{align}
for all $\spm\in\setC\setminus\{0\}$.
\end{theorem}
\begin{proof}
Let $\spm\in\setC\setminus\{0\}$. It is straightforward to see
$T_{\epsilon_0}^{h(n)}(\spm) \in L(X^{h(n)})$. Since $X^{h(n)}$ is finite dimensional,
$T_{\epsilon_0}^{h(n)}(\spm)$ is Fredholm with index zero. Further, we note that if
$n\in\setN$, $x_0\in\Omega$, $B_{R^*h(n)}(x_0)\subset\Omega$ and $u,\hat u\in X$ with
$u|_{B_{R^*h(n)}(x_0)}=\hat u|_{B_{R^*h(n)}(x_0)}$, then also
$(T_{\epsilon_0}^{h(n)}(\spm) u)|_{B_{h(n)}(x_0)}=(T_{\epsilon_0}^{h(n)}(\spm) \hat u)|_{B_{h(n)}(x_0)}$.
Indeed from Assumption~\ref{ass:LocalApproximation} follows
\begin{align*}
\|\Pi_{h(n)}(\eta_{\epsilon_0}(u-\hat u))\|_{H^1(B_{h(n)}(x_0))}
&=\|\eta_{\epsilon_0}(u-\hat u)-\Pi_h(\eta_{\epsilon_0}(u-\hat u))\|_{H^1(B_{h(n)}(x_0))}\\
&\leq C_1 \|\eta_{\epsilon_0}(u-\hat u)\|_{H^1(B_{C_1h(n)}(x_0))}
=0.
\end{align*}
So let $\hat r_1, \hat r_2$ be as in Lemma~\ref{lem:EtaEps}.
Let $r_2^*>\hat r_2$, $h_0>0$ with $h_0<\min\{\hat r_1-r^*_1, r_2^*-\hat r_2\}/C_1$
and $n_0>0$ be so that $h(n)<h_0$ for all $n>n_0$.
Let $n>n_0$ and $u_n\in X^{h(n)}$.
Since $\Pi_{h(n)}$ is linear and a projection it follows
\begin{align*}
(T_{\epsilon_0}^{h(n)}(\spm) u_n)|_{\Omega\cap B_{r^*_1}}
=\big(\Pi_{h(n)}(\eta_{\epsilon_0} u_n)\big)|_{\Omega\cap B_{r^*_1}}
&=\big(\Pi_{h(n)}(\eta_{\epsilon_0}(r_1^*) u_n)\big)|_{\Omega\cap B_{r^*_1}}\\
&=(\eta_{\epsilon_0}(r_1^*)\Pi_{h(n)} u_n)|_{\Omega\cap B_{r^*_1}}\\
&=(\eta_{\epsilon_0}(r_1^*) u_n)|_{\Omega\cap B_{r^*_1}}\\
&=(\eta_{\epsilon_0} u_n)|_{\Omega\cap B_{r^*_1}}
=(T_{\epsilon_0}(\spm) u_n)|_{\Omega\cap B_{r^*_1}}.
\end{align*}
Likewise $(T_{\epsilon_0}^{h(n)}(\spm) u_n)|_{\Omega\cap B_{r^*_2}^c}
=(T_{\epsilon_0}(\spm) u_n)|_{\Omega\cap B_{r^*_2}^c}$.
Hence
\small
\begin{align*}
&\|(T_{\epsilon_0}(\spm)-T_{\epsilon_0}^{h(n)}(\spm))u_n\|_{X}^2\\
&=
\langle (|\tilde d^2d^{-1}|\Px+|d|(\I-\Px))\nabla (T_{\epsilon_0}(\spm)-T_{\epsilon_0}^{h(n)}(\spm))u_n,
\nabla (T_{\epsilon_0}(\spm)-T_{\epsilon_0}^{h(n)}(\spm))u_n\rangle_{L^2(\Omega)}\\
&+\langle |\tilde d^2d|(T_{\epsilon_0}(\spm)-T_{\epsilon_0}^{h(n)}(\spm))u_n,(T_{\epsilon_0}(\spm)-T_{\epsilon_0}^{h(n)}(\spm))u_n \rangle_{L^2(\Omega)}\\
&=
\langle (|\tilde d^2d^{-1}|\Px+|d|(\I-\Px))\nabla (T_{\epsilon_0}(\spm)-T_{\epsilon_0}^{h(n)}(\spm))u_n,
\nabla (T_{\epsilon_0}(\spm)-T_{\epsilon_0}^{h(n)}(\spm))u_n\rangle_{L^2(A_{r_1^*,r_2^*})}\\
&+\langle |\tilde d^2d|(T_{\epsilon_0}(\spm)-T_{\epsilon_0}^{h(n)}(\spm))u_n,(T_{\epsilon_0}(\spm)-T_{\epsilon_0}^{h(n)}(\spm))u_n \rangle_{L^2(A_{r_1^*,r_2^*})}\\
&\leq C^2 \|(T_{\epsilon_0}(\spm)-T_{\epsilon_0}^{h(n)}(\spm))u_n\|_{H^1(A_{r^*_1,r^*_2})}^2
\end{align*}
\normalsize
with $C^2:=\sup_{x\in A_{r^*_1,r^*_2}} \max \{|\tilde d^2d^{-1}|,
|d|, |\tilde d^2d|\}<\infty$.
Now we are in the position to apply the analysis of Bertoluzza~\cite{Bertoluzza:99}.
Although we cannot apply~\cite[Theorem~2.1]{Bertoluzza:99} directly since neither
has $\eta_{\epsilon_0}$ compact support in $\Omega$ nor is the constant function included
in $X^{h(n)}$ (due to the incorporated homogeneous Dirichlet boundary condition).
Nevertheless, we can repeat the proof of~\cite[Theorem~2.1]{Bertoluzza:99} line
by line as follows.\newline

Let $h_0>0$ be so that $A_{r^*_1-R^*h_0,r^*_2+R^*h_0} \subset \Omega$
(with $R^*$ as in Assumption~\ref{ass:LocalApproximation}) and let $n_0>0$
be so that $h(n)<h_0$ for all $n>n_0$.
For each $n\in\setN$, $n>n_0$ we consider a collection of balls $\{B_{h(n)}(x), x\in Z\}$
with $Z\subset A_{r^*_1,r^*_2}$ so that $A_{r^*_1,r^*_2} \subset \bigcup_{x\in Z} B_{h(n)}(x)$
and so that any point $y\in \Omega$ belongs to at most $m\in\setN$ (with $m$ independent of $n\in\setN$, $n>n_0$)
balls of the collection $\{B_{R^*h(n)}(x), x\in Z\}$.
This implies the existence of a constant $\tilde C_1>0$ so that
\begin{align*}
\sum_{x\in Z} \|u\|_{H^s(B_{h(n)}(x))}^2 \leq \tilde C_1 \|u\|_{H^s(\bigcup_{x\in Z}B_{h(n)}(x))}^2
\end{align*}
for $s\in \{0,1,2\}$ and all $u\in H^s(\bigcup_{x\in Z}B_{h(n)}(x))$, $n\in\setN$, $n>n_0$.
Hence for $u_n\in X^{h(n)}$ we estimate
\begin{align*}
\|(T_{\epsilon_0}(\spm)-T_{\epsilon_0}^{h(n)}(\spm))u_n\|_{H^1(A_{r^*_1,r^*_2})}^2
&= \|(1-\Pi_{h(n)})\eta_{\epsilon_0}u_n\|_{H^1(A_{r^*_1,r^*_2})}^2\\
&\leq \sum_{x\in Z} \|(1-\Pi_{h(n)})\eta_{\epsilon_0}u_n\|_{H^1(B_{h(n)}(x))}^2.
\end{align*}
For each $x\in Z$ Assumption~\ref{ass:ConstantFunctions} allows us to
appropriately choose $u_{x,n}\in X^{h(n)}$ so that $u_{x,n}|_{B_{R^*h(n)}(x)}$ is
constant,
\begin{align*}
\|u_{x,n}\|_{L^2(B_{R^*h(n)}(x))} \leq \|u_n\|_{L^2(B_{R^*h(n)}(x))}
\end{align*}
and
\begin{align*}
\|u_n-u_{x,n}\|_{H^1(B_{R^*h(n)}(x))} \leq \tilde C_2 R^* h(n) \|u_n\|_{H^1(B_{R^*h(n)}(x))}
\end{align*}
with a constant $\tilde C_2>0$ independent of $u_n\in X^{h(n)}$ and $n\in\setN$, $n>n_0$.
Thus we estimate further
\begin{align*}
\|(1-\Pi_{h(n)})\eta_{\epsilon_0}u_n\|_{H^1(B_{h(n)}(x))}
&\leq \|(1-\Pi_{h(n)})\eta_{\epsilon_0}(u_n-u_{x,n})\|_{H^1(B_{h(n)}(x))}\\
&+ \|(1-\Pi_{h(n)})\eta_{\epsilon_0}u_{x,n}\|_{H^1(B_{h(n)}(x))}.
\end{align*}
Since $u_{x,n}|_{B_{R^*h(n)}(x)}$ is constant it follows with Assumption~\ref{ass:LocalApproximation}
\begin{align*}
\|(1-\Pi_{h(n)})\eta_{\epsilon_0}u_{x,n}\|_{H^1(B_{h(n)}(x))}
&\leq C_1 h(n) \|\eta_{\epsilon_0}u_{x,n}\|_{H^2(B_{R^*h(n)}(x))} \\
&\leq C_1 h(n) \|\eta_{\epsilon_0}\|_{W^{2,\infty}(B_{R^*h(n)}(x))}
\|u_{x,n}\|_{L^2(B_{R^*h(n)}(x))}.
\end{align*}
On the other hand, since $(u_n-u_{x,n})\in X^{h(n)}$ and $\Pi_{h(n)}$ is a projection
onto $X^{h(n)}$ it follows that $(1-\Pi_{h(n)})\eta_{\epsilon_0}(x)(u_n-u_{x,n})=0$.
Together with Assumption~\ref{ass:LocalApproximation} we estimate
\begin{align*}
\|(1-\Pi_{h(n)})\eta_{\epsilon_0}(u_n-u_{x,n})\|_{H^1(B_{h(n)}(x))}
&= \|(1-\Pi_{h(n)})(\eta_{\epsilon_0}-\eta_{\epsilon_0}(x))(u_n-u_{x,n})\|_{H^1(B_{h(n)}(x))}\\
&\leq C_1 \|(\eta_{\epsilon_0}-\eta_{\epsilon_0}(x))(u_n-u_{x,n})\|_{H^1(B_{R^*h(n)}(x))}.
\end{align*}
We compute
\begin{align*}
\|(\eta_{\epsilon_0}-\eta_{\epsilon_0}(x))(u_n-u_{x,n})\|_{H^1(B_{R^*h(n)}(x))}^2
&\leq \|(\eta_{\epsilon_0}-\eta_{\epsilon_0}(x))(u_n-u_{x,n})\|_{L^2(B_{R^*h(n)}(x))}^2\\
&+ 2\|(\eta_{\epsilon_0}-\eta_{\epsilon_0}(x))\nabla u_n)\|_{L^2(B_{R^*h(n)}(x))}^2\\
&+ 2\|(\nabla \eta_{\epsilon_0})(u_n-u_{x,n})\|_{L^2(B_{R^*h(n)}(x))}^2
\end{align*}
and estimate
\begin{align*}
\|(\eta_{\epsilon_0}-\eta_{\epsilon_0}(x))(u_n-u_{x,n})\|_{L^2(B_{R^*h(n)}(x))}
&\leq R^* h(n) \|\eta_{\epsilon_0}\|_{W^{1,\infty}(\Omega)} \|u_n\|_{L^2(B_{R^*h(n)}(x))},\\
\|(\eta_{\epsilon_0}-\eta_{\epsilon_0}(x))\nabla u_n\|_{L^2(B_{R^*h(n)}(x))}
&\leq R^* h(n) \|\eta_{\epsilon_0}\|_{W^{1,\infty}(\Omega)} \|u_n\|_{H^1(B_{R^*h(n)}(x))},\\
\|(\nabla \eta_{\epsilon_0})(u_n-u_{x,n})\|_{L^2(B_{R^*h(n)}(x))}
&\leq \tilde C_2 R^* h(n) \|\eta_{\epsilon_0}\|_{W^{1,\infty}(\Omega)} \|u_n\|_{H^1(B_{R^*h(n)}(x))}.
\end{align*}
Altogether we obtain
\begin{align*}
\|(T_{\epsilon_0}(\spm)-T_{\epsilon_0}^{h(n)}(\spm))u_n\|_{H^1(A_{r^*_1,r^*_2})}
\leq \tilde C_3 h(n) \|u_n\|_{H^1(A_{r^*_1-R^*h_0,r^*_2+R^*h_0})}
\end{align*}
with a constant $\tilde C_3>0$ independent of $n\in\setN$, $n>n_0$, $u_n\in X^{h(n)}$.
It remains to note
\begin{align*}
\|u_n\|_{H^1(A_{r^*_1-R^*h_0,r^*_2+R^*h_0})} \leq \tilde C_4 \|u_n\|_X
\end{align*}
for a constant $\tilde C_4>0$ independent of $n\in\setN$, $n>n_0$, $u_n\in X^{h(n)}$.
\end{proof}

\begin{theorem}[Spectral convergence]\label{thm:SpectralConvergenceDirect}
Let Assumptions~\ref{ass:TildeAlpha} and \ref{ass:TildeAlpha2} hold.
Let $X$ and $a(\cdot;\cdot,\cdot)$ be as in~\eqref{eq:X-sprX-a}, $A(\cdot)$ be as in~\eqref{eq:Ak}
and $\Lambda_{d_0}^\pm$ be as in~\eqref{eq:Lambdad0pm}.
Let $\big(X^{h(n)}\big)_{n\in\setN}$ be sequence of finite dimensional subspaces
$X^{h(n)}\subset X$ so that the orthogonal projections from $X$ onto $X^{h(n)}$
converge point-wise to the identity in $X$ and so that Assumptions~\ref{ass:LocalApproximation}
and~\ref{ass:ConstantFunctions} hold. Let $A^{h(n)}(\cdot)$ be defined by~\eqref{eq:Akn}
and $T_{\epsilon_0}(\cdot)$ be as in Theorem~\ref{thm:Tepsh}.

Then $A(\cdot)\colon \Lambda_{d_0}^\pm\to L(X)$ is a weakly $T_{\epsilon_0}(\cdot)$-coercive
holomorphic Fredholm operator function with non-empty resolvent set $\rho\big(A(\cdot)\big)$,
$A^{h(n)}(\cdot)\colon \Lambda_{d_0}^\pm\to L(X^{h(n)})$ is a $T_{\epsilon_0}(\cdot)$-compatible approximation and hence the convergence results \ref{item:SP-a})-\ref{item:SP-g}) from Subsection~\ref{subsec:Tframework} hold.
\end{theorem}
\begin{proof}
Follows from Theorem~\ref{thm:SpectralConvergenceHelmholtz} and Theorem~\ref{thm:Tepsh}.
\end{proof}

All three assumptions are fulfilled by common finite element spaces, see e.g.\
\cite{BrennerScott:08}. By means of Lemma~\ref{lem:BestApproximation} and the triangle inequality we can obtain convenient error bounds (for finite element methods) of the form
\begin{align*}
\|u-u_n\|_X \lesssim e^{-Cr_n}+h^p
\end{align*}
for some constant $C>0$ with layer width $r_n$, mesh width $h$ and polynomial degree $p$.

\section{Truncationless approximation}\label{sec:appr_truncationless}

As previously discussed, the classical approach to approximate~\eqref{eq:ResonanceProblemVariational}
is to first choose a bounded subdomain $\Omega_n\subset\Omega$ and secondly
to choose a convenient Galerkin space $X^h\subset H^1_0(\Omega_n)$, e.g.\ a finite
element space. However, if the approximation is not satisfactory enough and a better
approximation is desired, it is in general not enough to increase the dimension
of the finite element space, but also the size of the domain $\Omega_n$ needs to
be increased. The latter involves a new domain and the generation of a new mesh.
This may be undesirable for people who would prefer to work with a fixed domain
and solely increase discretization parameters in order to avoid a new meshing process. There are at least two concepts to achieve this goal.

One is the implementation of infinite elements into the code. I.e.\ the fixed
domain is $\Omega\cap B_{r_1^*}$ and the exterior domain $\Omega\setminus B_{r_1^*}$
is not explicitly meshed. Instead tensor product (finite element) functions with
respect to polar coordinates can be used. This can indeed be implemented without
the explicit generation of a mesh for $\Omega_n\setminus B_{r_1^*}$.
Of course it is possible to also use non-classical basis functions with respect
to the radial variable, e.g.\ as $\exp(-r)p(r)$ with polynomials $p$. We mention the recent work~\cite{NannenWess:19}
wherein the Hardy space infinite element method introduced in~\cite{HohageNannen:09} is framed as a complex scaling
infinite element method. We note that the analysis thereof is already covered by \cite{Halla:16}.

A different approach is to derive a formulation of the eigenvalue problem which
involves only a bounded domain (but singular coefficients) and subsequently to
apply a classical finite element discretization.
To our knowledge Berm\'udez et.\ al.\ were the first to consider
a variant of this idea in \cite{BermudezHervellaNPrietoRodriguez:04} and subsequently
in \cite{BermudezHervellaNPrietoRodriguez:08}. Their idea is to use a profile function
$\tilde\alpha$ which is unbounded on $(0,r_2^*)$ with $r_2^*>r_1^*$. This leads
to a formulation of the eigenvalue problem on the bounded domain $\Omega\cap B_{r_2^*}$.
Since in this case the formulation (and subsequently the discretization) is posed
on a bounded domain without committing a truncation error, Berm\'udez et.\ al.\ coined their
method ``exact PML''.
Another variant is to consider the formulation as derived in Section~\ref{sec:theproblem}
and subsequently apply a \emph{real} domain transformation $B_{r_1^*}^c\to A_{r_1^*,r_2^*}$, which is essentially the approach of Hugonin and Lalanne \cite{HugoninLalanne:05}.
There is a noteworthy alternative interpretation to both methods~\cite{Nannen:pc16}.
Namely the formulation can be transformed (back) to the unbounded domain $\Omega$.
If this happens after the discretization one obtains a discretization of the problem
posed in $\Omega$. This way one implicitly applies basis functions with unbounded support.
Thus these mentioned ``exact'' methods could also validly be called ``infinite element''
methods. However, we will stick to the formulations on bounded domains for convenience.

A difference between these two methods is that the method based on the real domain transformation still allows the choice of $d_0$ and hence a control of the essential spectrum $\{d_0^{-1}x\colon x\in\setR\}$, while for the method of Berm\'udez et.\ al.\ the essential spectrum is implicitly set to $\{-ix\colon x\in\setR\}$. This is of importance if one seeks to apply these techniques to problems which involve evanescent waves which occur e.g.\ for waveguide geometries. The former technique can be applied successfully to such problems, while the latter technique fails.

In the following, we consider only the method motivated by \cite{HugoninLalanne:05} and refer to \cite{Halla:19Diss} for a similar discussion on the method of \cite{BermudezHervellaNPrietoRodriguez:08}.
We derive from Eigenvalue Problem~\eqref{eq:ResonanceProblemVariational} by means of a real domain transformation $\exx$ (see~\eqref{eq:exx} and Assumption~\ref{ass:DomainTransformation}) the related Eigenvalue Problem~\eqref{eq:ResonanceProblemVariationalExact} and perform an approximation analysis.
The analysis involves no new concepts but only slightly adapts the techniques of
the previous sections, in particular the technique of Section~\ref{sec:appr_simul}.
Finally we discuss how appropriate finite element spaces fit the derived theory. \newline

We consider real domain transformations $\exr$ of the following kind.
\begin{assumption}\label{ass:DomainTransformation}
Let $r_1^*$ be as in Assumption~\ref{ass:TildeAlpha} and $r_2^*>r_1^*$.
Let $\exr\colon(0,r_2^*)\to\setR^+$ be bijective, continuous, $\exr|_{(r_1^*,r_2^*)}$ be
continuously differentiable and so that $\exr(r)=r$ for $r\leq r^*_1$.
\end{assumption}

Let $\tilde\alpha$ suffice Assumption~\ref{ass:TildeAlpha} and Assumption~\ref{ass:TildeAlpha2}.
Let $\exr$ suffice Assumption~\ref{ass:DomainTransformation}, $\tilde d$, $\tilde r$,
$\alpha$, $d$ be as in \eqref{eq:ComplexScalingQuantities}, $\hat d$ be as in~\eqref{eq:Hatd} and
\begin{subequations}\label{eq:ComplexScalingQuantitiesExact}
\begin{align}
\label{eq:exx}\exx(x)&:=\exr(|x|)/|x|x,\\
\exgamma(x)&:=(\partial_r\exr)(|x|),\\
\extgamma(x)&:=\exr(|x|)/|x|,\\
\extalpha&:=\tilde\alpha\circ\exr,\\
\extd&:=\tilde d \circ\exr,\\
\extr&:=\tilde r\circ\exr,\\
\exalpha&:=\alpha \circ\exr,\\
\exd&:=d \circ\exr,\\
\exhatd&:=\hat d\circ\exr.
\end{align}
\end{subequations}
As hitherto we adopt the overloaded notation~\eqref{eq:OverloadedNotation} also for
the new quantities $\exr$, $\extr$, $\exgamma$, $\extgamma$, $\extalpha$, $\exalpha$, $\extd$, $\exd$, $\exhatd$.
We compute
\begin{subequations}
\begin{align}
\D \exx &= \exgamma\Px+\extgamma(\I-\Px),\\
(\D \exx)^{-1} &= \exgamma^{-1}\Px+\extgamma^{-1}(\I-\Px),\\
\det \D \exx &= \exgamma\extgamma^2.
\end{align}
\end{subequations}
We consider the bounded domain
\begin{align}
\exOmega&:=\Omega\cap B_{r_2^*},
\end{align}
subsequently set
\begin{subequations}\label{eq:exaX}
\begin{align}
\begin{aligned}\label{eq:exa}
\exa(\spm; u,u')&:=\langle \extgamma^2\exgamma^{-1}\extd^2\exd^{-1}\Px
+ \exgamma\exd(\I-\Px))\nabla u,\nabla u'\rangle_{L^2(\exOmega)}\\
&-\spm^2\langle \extgamma^2\exgamma\extd^2\exd u,u' \rangle_{L^2(\exOmega)},
\end{aligned}
\end{align}
\begin{align}
\label{eq:exX}
\exX&:=\{u\in H^1_\mathrm{loc}(\exOmega)\colon \langle u,u\rangle_{\exX}<\infty, u|_{\partial\Omega}=0 \},\\
\langle u,u' \rangle_{\exX}&:=\langle u,u' \rangle_{\exX(\exOmega)},
\end{align}
and
\begin{align}
\begin{aligned}
\langle u,u' \rangle_{\exX(D)}&:=\langle (\extgamma^2\exgamma^{-1}|\extd^2\exd^{-1}|\Px+\exgamma|\exd|(\I-\Px))\nabla u,
\nabla u'\rangle_{L^2(D)}\\
&+\langle \extgamma^2\exgamma|\extd^2\exd|u,u' \rangle_{L^2(D)},
\end{aligned}
\end{align}
\end{subequations}
for $\spm\in\setC$, $u,u'\in \exX$ and $D\subset\exOmega$ and consider the eigenvalue problem to
\begin{align}\label{eq:ResonanceProblemVariationalExact}
\text{find }(\spm,u)\in\setC\times \exX\setminus\{0\}\text{ so that}\quad \exa(\spm;u,u')=0
\quad\text{for all }u'\in \exX.
\end{align}
Due to the transformation rule and the chain rule it is clear that
\begin{align}
\exF u:= u\circ\exx
\end{align}
is a linear bijective Hilbert space isomorphism, i.e.\ $\exF\in L(X,\exX)$,
$\exF$ is bijective and
\begin{align}
\spl u,u'\spr_X = \spl \exF u, \exF u' \spr_{\exX}
\end{align}
for all $u, u'\in X$ (with $X$ as in \eqref{eq:X-sprX-a}). Further it holds
\begin{align}
a(\spm;u,u')=\exa(\spm;\exF u,\exF u') 
\end{align}
for all $u,u'\in X$. Thus we can simply deduce the properties of $\exA(\cdot)$ (defined through~\eqref{eq:Ak}) from $A(\cdot)$.
In particular it holds that $(\spm,u)\in \setC \times X\setminus\{0\}$ is a solution to $A(\spm)u=0$ if and only if $\exA (\spm)\exF u=0$. $\exA (\spm)$ is Fredholm
if and only if $\spm\in\Lambda_{d_0}$ (with $\Lambda_{d_0}$ as in~\eqref{eq:Lambdad0}).
Further $\exA (\cdot)|_{\Lambda_{d_0}}$ is weakly $\exT(\cdot)$-coercive with
\begin{align}
\exT(\spm) u = \exF T(\spm) \exF^{-1} u = (\eta\circ\exx) u
\end{align}
for $u\in\exX$ and $\eta$ being the symbol of $T(\spm)$ as in~\eqref{eq:DefTHH}.
Further $\exA (\spm)$ is bijective for all $\spm\in\setC\setminus\{0\}$ with
$\arg\spm\in [-\pi,-\arg d_0)\cup[0,\pi-\arg d_0)$. \newline

It remains to discuss the approximation of~\eqref{eq:ResonanceProblemVariationalExact}.
Hence we first adapt Lemma~\ref{lem:WeigthedL2Compactness} to our current setting in
Lemma~\ref{lem:WeigthedL2CompactnessExact}. Then we proceed as in Section~\ref{sec:appr_simul} and construct an operator function $\exTeps(\cdot)$ with appropriate properties.

\begin{lemma}\label{lem:WeigthedL2CompactnessExact}
Let $(r_n)_{n\in\setN}$ with $r_n\in(r_1^*,r_2^*)$ for all $n\in\setN$ be a monotonically increasing sequence with limes $r_2^*$.
Let $\eta_1\colon\exOmega\to\setC$ be mesuarable so that $\eta_1|_{\exOmega\cap B_{r_n}}$ $\in$ $L^\infty(\exOmega\cap B_{r_n})$
for all $n\in\setN$. Let $Y\subset L^2(\exOmega)$ be a Hilbert space so that $\|\eta_1 u\|_{L^2(\exOmega)}\leq C\|u\|_Y$
with $C>0$ for all $u\in Y$
and so that the embedding and restriction operator $K_n\colon Y\to L^2(\exOmega\cap B_{r_n})
\colon u\mapsto u|_{\exOmega\cap B_{r_n}}$ is compact for each $n\in\setN$.
Let $\eta_2\in L^\infty(\exOmega)$ be so that $\lim_{r\to r_2^*-}\|\eta_2\|_{L^\infty(\exOmega\setminus B_r)}=0$.
Then the multiplication and embedding operator $K_{\eta_1\eta_2}\colon Y\to L^2(\exOmega)
\colon u\mapsto \eta_1\eta_2 u$ is compact.
\end{lemma}
\begin{proof}
Proceed as in the proof of Lemma~\ref{lem:WeigthedL2Compactness}.
\end{proof}

\begin{lemma}\label{lem:TepsExact}
Let Assumptions~\ref{ass:TildeAlpha}, \ref{ass:TildeAlpha2} and \ref{ass:DomainTransformation} hold.
Let $\exX$ be as in~\eqref{eq:exX}, $\exa(\cdot;\cdot,\cdot)$ be as in~\eqref{eq:exa} and
$\exA(\cdot)$ be as in~\eqref{eq:Ak}. For $\epsilon>0$ and $\spm\in\setC\setminus\{0\}$ let
\begin{align}\label{eq:DefTepsExact}
\exTeps(\spm)u:=\exetaeps u \qquad\text{with} \qquad
\exeta=\left\{\begin{array}{ll}
\ol{\frac{|\exhatd|}{\exhatd}} &\text{for } \arg(-\spm^2d_0^2)\in [-\pi,0), \vspace{3mm}\\
\ol{\frac{\exhatd}{\extd^2}\frac{|\extd|^2}{|\exhatd|}} &\text{for }\arg (-\spm^2d_0^2)\in [0,\pi).
\end{array}\right.
\end{align}
with $\exetaeps|_{(r_1^*,r_2^*)}$ as in Lemma~\ref{lem:EtaEps} with $r_1=r_1^*, r_2=r_2^*$
and $\exetaeps|_{[0,r_1^*]}:=\exetaeps(r_1^*)$.

There exists $\epsilon_0(\omega)>0$ so that for each $\epsilon\leq\epsilon_0(\omega)$,
$\exTeps(\spm) \in L(\exX)$ is bijective for all $\spm\in\setC\setminus\{0\}$
and $\exA(\cdot)\colon \Lambda_{d_0}\to L(\exX)$ is weakly $\exTeps(\cdot)$-coercive.
\end{lemma}
\begin{proof}
Proceed as in the proof of Lemma~\ref{lem:Teps} with Lemma~\ref{lem:WeigthedL2Compactness} replaced by Lem.~\ref{lem:WeigthedL2CompactnessExact}.
\end{proof}

Next we consider a sequence of finite dimensional subspaces $(\exXh)_{n\in\setN}$,
$\exXh \subset \exX$, $n\in\setN$ so that the orthogonal projections onto $\exXh$
converge point-wise to the identity in $\exX$.
Further let
\begin{align}\label{eq:ResonanceProblemGalerkinExact}
\text{find }(\spm,u)\in\setC\times \exXh\setminus\{0\}\text{ so that}\quad \exa(\spm;u,u')=0
\quad\text{for all }u' \in \exXh
\end{align}
be the Galerkin approximation to~\eqref{eq:ResonanceProblemVariationalExact}.
As in Section~\ref{sec:appr_simul} we make two additional assumptions on the
Galerkin spaces $\exXh$.

\begin{assumption}\label{ass:LocalApproximationExact}
There exists a sequence $\big(h(n)\big)_{n\in\setN}\in (\setR^+)^\setN$ with
$\lim_{n\in\setN} h(n)=0$.
There exist bounded linear projection operators $\exPin\colon \exX\to \exXh, n\in\setN$
that act locally in the following sense: there exist constants $C_1, R^*>1$
so that for $n\in\setN$, $s \in \{1,2\}$, $x_0\in\exOmega$,
if $B_{R^*h(n)}(x_0)\subset \exOmega$, $u\in \exX$ and $u|_{B_{R^*h(n)}(x_0)}
\in H^s(B_{R^*h(n)}(x_0))$, then
\begin{align}
\|u-\exPin u\|_{H^1(B_{h(n)}(x_0))} \leq C_1 h(n)^{s-1} \|u\|_{H^s(B_{R^*h(n)}(x_0))}.
\end{align}
\end{assumption}

\begin{assumption}\label{ass:ConstantFunctionsExact}
For any $D\subset\exOmega$ which is compact in $\exOmega$ exists $n_0>0$
so that for each $n\in\setN, n>n_0$ there exists $u_{D,n}\in \exXh$ with $u_{D,n}|_D=1$.
\end{assumption}

\begin{theorem}\label{thm:TepshExact}
Let Assumptions~\ref{ass:TildeAlpha}, \ref{ass:TildeAlpha2} and \ref{ass:DomainTransformation} hold.
Let $\exX$ be as in~\eqref{eq:exX}, $\big(\exXh)_{n\in\setN}$ be a sequence of
finite dimensional subspaces $\exXh\subset \exX$ so that the orthogonal projections
onto $\exXh$ converge point-wise to the identity and so that Assumptions~\ref{ass:LocalApproximationExact}
and~\ref{ass:ConstantFunctionsExact} hold. Let $\exTepsz(\omega):=T_{\inde,\epsilon_0(\omega)}(\omega)$ be as
in Lemma~\ref{lem:TepsExact} and $\|\cdot\|_n$ be as in~\eqref{eq:discretenorm}. For $n\in\setN$ let $\exPin$ be as in Assumptions~\ref{ass:LocalApproximationExact} and
\begin{align}
\exTepszn:=\exPin \exTepsz(\spm)|_{\exXh}
\end{align}
for $\spm\in\setC\setminus\{0\}$.
Then $\exTepszn(\spm) \in L(\exXh)$ is Fredholm with index zero and
\begin{align}
\lim_{n\in\setN} \|\exTepsz(\spm)-\exTepszn(\spm)\|_n=0
\end{align}
for all $\spm\in\setC\setminus\{0\}$.
\end{theorem}
\begin{proof}
Proceed as in the proof of Theorem~\ref{thm:TepshExact}.
\end{proof}

\begin{theorem}[Spectral convergence]\label{thm:SpectralConvergenceHelmholtzExact}
Let Assumptions~\ref{ass:TildeAlpha} and \ref{ass:TildeAlpha2} hold.
Let $\exr$ fulfill Assumption~\ref{ass:DomainTransformation} and $\exX$, $\exa(\cdot;\cdot,\cdot)$ be as defined in~\eqref{eq:exaX}.
Let $\exA(\cdot)\colon\Lambda\to L(\exX)$ be defined through~\eqref{eq:Ak}, $\exTepsz(\cdot)$ as in Theorem~\ref{thm:TepshExact} and $\Lambda_{d_0}^\pm$ be as in~\eqref{eq:Lambdad0pm}.
Let $\big(\exXh\big)_{n\in\setN}$ be a sequence of finite dimensional subspaces $\exXh\subset \exX$ so that the orthogonal projections from $\exX$ onto $\exXh$ converge point-wise to the identity in $\exX$ and so that Assumptions~\ref{ass:LocalApproximationExact} and~\ref{ass:ConstantFunctionsExact} hold.
Let $\exAh(\cdot)$ be defined by~\eqref{eq:Akn} and $\exTepszn(\cdot)$ be as in
Theorem~\ref{thm:TepshExact}.

Then $\exA(\cdot)\colon \Lambda_{d_0}^\pm\to L(\exX)$ is a weakly $\exTepsz(\cdot)$-coercive
holomorphic Fredholm operator function with non-empty resolvent set $\rho\big(\exA(\cdot)\big)$
and $\exAh(\cdot)\colon \Lambda_{d_0}^\pm\to L(\exXh)$ is a $\exTepsz(\cdot)$-compatible
approximation, i.e.\ the convergence results \ref{item:SP-a})-\ref{item:SP-g}) from Subsection~\ref{subsec:Tframework} hold.
\end{theorem}
\begin{proof}
Proceed as in the proof of Theorem~\ref{thm:SpectralConvergenceDirect}.
\end{proof}

Finally we discuss how to choose appropriate parameters $\tilde\alpha$, $\exr$ and an appropriate sequence of subspaces
$(\exXh)_{n\in\setN}$, $\exXh\subset\exX$. To this end we introduce two lemmata.

\begin{lemma}\label{lem:exXhDense}
Assume that
\begin{align}\label{eq:AssForXhDense}
\sup_{x\in\exOmega} \frac{1}{(r_2^*-|x|)\exgamma(x)|\exd(x)|}<+\infty.
\end{align}
Let $(\exXh)_{n\in\setN}$, $\exXh\subset\exX$ be so that for any $\delta>0$ and $u\in\exX$ with
$u|_{A_{r_2^*-\delta,r_2^*}}=0$ it holds
\begin{align}\label{eq:InfApproxError}
\lim_{n\in\setN} \inf_{u'\in\exXh} \|u-u'\|_{\exX}=0.
\end{align}
Then~\eqref{eq:InfApproxError} holds for any $u\in\exX$.
\end{lemma}
\begin{proof}
For $\delta>0$ consider
\begin{align*}
g_\delta(x):=\cofii\big(|x|/\delta -(r_2^*-2\delta)/\delta\big)
\end{align*}
with $\cofii$ as in~\eqref{eq:Smooth0to1}. Let $u\in\exX$ and $\epsilon>0$ be given. By means of the product rule,
the triangle inequality, the properties of $g_\delta$ and the chain rule we compute
\begin{align*}
\|g_\delta u\|_{\exX} &\leq
2\langle \extgamma^2\exgamma^{-1}|\extd^2\exd^{-1}|\Px u\nabla g_\delta,u\nabla g_\delta\rangle_{L^2(\exOmega)}\\
&+2\langle \extgamma^2\exgamma^{-1}|\extd^2\exd^{-1}| g_\delta^2 \Px\nabla u,\nabla u\rangle_{L^2(\exOmega)}\\
&+\langle \exgamma|\exd|g_\delta^2 (\I-\Px))\nabla u,\nabla u\rangle_{L^2(\exOmega)}\\
&+\langle \extgamma^2\exgamma|\extd^2\exd| g_\delta^2 u,u \rangle_{L^2(\exOmega)}\\
&\leq
2 \Big(1+\big(\sup_{x\in A_{r_2^*-2\delta,r_2^*-\delta}} |\nabla g_\delta|^2(\exgamma\exd)^{-2}\big)
\Big)\|u\|^2_{\exX(A_{r_2^*-2\delta,r_2^*})}\\
&\leq
2 \Big(1+\|\partial_r \cofii\|_{L^\infty(0,1)}^2\big(\sup_{x\in A_{r_2^*-2\delta,r_2^*-\delta}} (\delta\exgamma\exd)^{-2}\big)
\Big)\|u\|^2_{\exX(A_{r_2^*-2\delta,r_2^*})}\\
&\leq
2 \Big(1+\|\partial_r \cofii\|_{L^\infty(0,1)}^2\big(\sup_{x\in A_{r_2^*-2\delta,r_2^*-\delta}}
((r_2^*-|\cdot|)\exgamma\exd)^{-2}\big)\Big)\|u\|^2_{\exX(A_{r_2^*-2\delta,r_2^*})}\\
&\leq
2 \Big(1+\|\partial_r \cofii\|_{L^\infty(0,1)}^2\big(\sup_{x\in\exOmega}
((r_2^*-|\cdot|)\exgamma\exd)^{-2}\big)\Big)\|u\|^2_{\exX(A_{r_2^*-2\delta,r_2^*})}\\
&=:C\|u\|^2_{\exX(A_{r_2^*-2\delta,r_2^*})}.
\end{align*}
Due to $\lim_{\delta\to0+}\|u\|^2_{\exX(A_{r_2^*-2\delta,r_2^*})}=0$ we can choose $\delta>0$ so that
$C\|u\|^2_{\exX(A_{r_2^*-2\delta,r_2^*})}$ $<$ $\epsilon/2$. Since $1-g_\delta(x)=0$ for $x\geq r_2^*-\delta$ we can choose $n_0\in\setN$ so that
\begin{align*}
\inf_{u'\in\exXh} \|(1-g_\delta)u-u'\|_{\exX}<\epsilon/2
\end{align*}
for all $n>n_0$. It follows for all $n>n_0$
\begin{align*}
\inf_{u'\in\exXh} \|u-u'\|_{\exX}
&\leq \|g_\delta u\|_{\exX} + \inf_{u'\in\exXh} \|(1-g_\delta)u-u'\|_{\exX}\\
&\leq \epsilon/2+\epsilon/2=\epsilon.
\end{align*}
Since $\epsilon>0$ was chosen arbitrarily it follows that $\lim_{n\in\setN} \inf_{u'\in\exXh} \|u-u'\|_{\exX}=0$.
\end{proof}

\begin{lemma}\label{lem:exXhSubspace}
Let $\tilde\alpha$ be of Kind~\eqref{eq:AlphaAffin} or~\eqref{eq:AlphaSmooth}.
Let $\exr$ be so that for $r\in(r^*_1,r_2^*)$ either
\begin{subequations}
\begin{align}\label{eq:exrlog}
\exr(r)=-(\ln(r_2^*-r)-\ln(r_2^*-r_1^*))+r_1^*
\end{align}
or
\begin{align}\label{eq:exrbeta}
\exr(r)=(r_2^*-r)^\beta-(r_2^*-r_1^*)^\beta+r_1^*
\end{align}
\end{subequations}
with $\beta\in(-2/3,0)$. Then~\eqref{eq:AssForXhDense} holds.
If $u\in H^1_0(\exOmega)$ is so that $|u(x)|\leq C(r_2^*-|x|)$ for a constant $C>0$ and all $x\in\exOmega$,
then $u\in\exX$.
\end{lemma}
\begin{proof}
For $\exr$ as in~\eqref{eq:exrlog} it holds $\exgamma(x)=(r_2^*-|x|)^{-1}$.
For $\exr$ as in~\eqref{eq:exrbeta} it holds $\exgamma(x)=-\beta(r_2^*-1)^{\beta-1}$.
Since $|\exd|\geq1$ it easiliy follows~\eqref{eq:AssForXhDense} in both cases.
Due to the choice of $\tilde\alpha$ the coefficients $|\extd^2/\exd|$, $|\exd|$, $|\extd^2\exd|$ are uniformly bounded.
For $\exr$ as in~\eqref{eq:exrlog} we compute
\begin{subequations}\label{eq:exrlogcoeff}
\begin{align}
\extgamma(x)^2/\exgamma(x) &= \big( -(\ln(r_2^*-|x|)-\ln(r_2^*-r_1^*))+r_1^* \big)^2|x|^{-2}(r_2^*-|x|),\\
\exgamma(x)(r_2^*-|x|)^2 &= (r_2^*-|x|),\\
\extgamma(x)^2\exgamma(x)(r_2^*-|x|)^2 &= \big( -(\ln(r_2^*-|x|)-\ln(r_2^*-r_1^*))+r_1^* \big)^2|x|^{-2}(r_2^*-|x|).
\end{align}
\end{subequations}
For $\exr$ as in~\eqref{eq:exrbeta} we compute
\begin{subequations}\label{eq:exrbetacoeff}
\begin{align}
\extgamma(x)^2/\exgamma(x) &= \big( (r_2^*-|x|)^\beta-(r_2^*-r_1^*)^\beta+r_1^* \big)^2|x|^{-2}
(-\beta)^{-1}(r_2^*-|x|)^{-\beta+1},\\
\exgamma(x)(r_2^*-|x|)^2 &= -\beta(r_2^*-|x|)^{\beta+1},\\
\extgamma(x)^2\exgamma(x)(r_2^*-|x|)^2 &= \big( (r_2^*-|x|)^\beta-(r_2^*-r_1^*)^\beta+r_1^* \big)^2|x|^{-2}
(-\beta)(r_2^*-|x|)^{\beta+1}.
\end{align}
\end{subequations}
It follows that each function in~\eqref{eq:exrlogcoeff} and \eqref{eq:exrbetacoeff} is uniformly bounded in
$x\in A_{r_1^*,r_2^*}$. It follows $\|u\|_{\exX}<+\infty$.
\end{proof}

Consider $\exr$ and $\tilde\alpha$ as in Lemma~\ref{lem:exXhSubspace}.
Due to Lemma~\ref{lem:exXhSubspace} common finite element spaces are indeed subspaces of $\exX$.
Due to Lemma~\ref{lem:exXhDense} $(\exXh)_{n\in\setN}$ is asymptotically dense in $\exX$ if it is so in $H^1_0(\exOmega)$.
Hence with the stated choice of parameters $\tilde\alpha$, $\exr$ a reliable discretization of \eqref{eq:ResonanceProblemVariationalExact} can be constructed straightforwardly.

\section{Conclusion}\label{sec:conclusion}
We introduced a new abstract framework to analyze complex scaling/perfectly matched layer approximations of resonance problems.
It requires rather minimal assumptions on the scaling profile and includes convergence rates. It also covers approximations through simultaneaous truncation and discretization, and also truncationless methods.

In this article we applied the framework to scalar resonance problems in homogeneous exterior domains.
We constructed the framework in such a way that it can be suitably adapted to serve also for other kinds of partial differential equations and geometrical setups.
In particular we plan to extend our results to electromagnetic and elastic equations, and to scalar equations in plates.
On the other hand, an application to cartesian scalings seems only partially possible, because in this case an \emph{explicit} T-operator (to achieve weak T-coercivity) is not known.
For the same reason an application to open waveguide geometries seems challenging.

At last we give some remarks on the perspective to develop error estimators and adaptive methods for CS/PML.
The interpretation of discretized truncations as conform Galerkin approximations opens a new door to this end.
The truncation error is proportional to the decay of $u$ at the truncation boundary, which can be measured locally with the norm of the Dirichlet trace of $u$.
Since the solution $u$ is not available it is replaced by the numerical solution $u_h$.
Because the Dirichlet trace of $u_h$ vanishes due to the enforced boundary condition, the Neumann trace of $u_h$ can be used instead.
This notion can be made rigorous. E.g.\ for scattering problems a residual error estimator woud lead to local error estimators $h_F^{1/2}\|\nu\cdot (\tilde d^2d^{-1}\Px+d(\I-\Px))\nabla u_h\|_{L^2(F)}$ at the artificial boundary to measure the truncation error.
This approach is kind of familiar to \cite{ChenLiu:05a} which uses the Dirichlet trace of $u_h$ at $S^2_{r_1^*}$ to measure the truncation error.
However, for both estimators a fully adaptive method would still require to increase the domain size and hence the domain size is usually chosen a priori large enough.
To avoid this, the truncationless method of Section~\ref{sec:appr_truncationless} obtrudes itself. Though, the (residual) error estimators would need to respect the weighted norm of the space.

\bibliographystyle{amsplain}
\bibliography{short_biblio}
\end{document}